\documentclass[12pt]{article}
\usepackage[a4paper, margin=1.0in]{geometry}
\usepackage[utf8]{inputenc}
\usepackage{amsfonts}
\usepackage{amssymb}
\usepackage{amsmath}
\usepackage{amsthm}
\usepackage{ifthen}
\usepackage{xifthen}
\usepackage{xfrac}
\usepackage[all, cmtip]{xy}
\usepackage{graphicx}
\usepackage[capitalise]{cleveref}
\usepackage{pdfpages}
\usepackage{tikz}
\usetikzlibrary{positioning}
\usepackage{mathtools}
\crefname{section}{§}{§§}
\Crefname{section}{§}{§§}

\theoremstyle{definition}
\newtheorem{thm}{Theorem}[section]
\newtheorem*{thm*}{Theorem}
\newtheorem{criterion}[thm]{Criterion}
\newtheorem{cor}[thm]{Corollary}
\newtheorem{claim}[thm]{Claim}
\newtheorem{prop}[thm]{Proposition}
\newtheorem{lem}[thm]{Lemma}
\newtheorem{dfn}[thm]{Definition}
\newtheorem{exmd}[thm]{Example}

\newcounter{tmp}
\newcommand\blfootnote[1]{%
  \begingroup
  \renewcommand\thefootnote{}\footnote{#1}%
  \addtocounter{footnote}{-1}%
  \endgroup
}

\newcommand{\acts}{\curvearrowright}
\newcommand{\act}{\boldsymbol{.}}
\makeatother

\newcommand{\al }{\alpha  }

\newcommand{\Ga }{\Gamma  }
\newcommand{\Del}{\Delta  }
\newcommand{\tht}{\theta  }
\newcommand{\la }{\lambda }
\newcommand{\eps}{\epsilon}
\newcommand{\vp }{\varphi }

\newcommand{\bb}[1]{\mathbb{#1}}
\renewcommand{\cal}[1]{\mathcal{#1}}

\newcommand{\GG}{G\times G}
\newcommand{\XX}{X\times X}
\newcommand{\Gm}[1][] { 
	\ifthenelse{\isempty{#1}} {
		\mathbb{G}_{m}
	} {
		\mathbb{G}_{m}\left(#1\right)
	}
}
\newcommand{\Gmt}[1][]{\widetilde{\Gm[#1]}}
\newcommand{\Aff}[2][] {
	\ifthenelse{\isempty{#1}} {
		\mathbb{A}^{#2}
	} {
		\mathbb{A}^{#2}_{#1}
	}
}
\newcommand{\GGt}{\widetilde{G\times G}}
\newcommand{\gl}[2][] {
	\ifthenelse{\isempty{#1}} {
		\mathfrak{gl}_{#2}
	} {
		\mathfrak{gl}_{#2}\left(#1 \right)
	}
}
\newcommand{\Pgl}[2][]{\mathbb{P}\gl[#1]{#2}}
\newcommand{\GL}[2][] {
	\ifthenelse{\isempty{#1}} {
		GL_{#2}
	} {
		GL_{#2}\left(#1\right)
	}
}
\newcommand{\PGL}[2][]{P\GL[#1]{#2}}
\newcommand{\ZZ}[1][] {
	\ifthenelse{\isempty{#1}} {
		\mathbb{Z}
	} {
		\sfrac{\mathbb{Z}}{#1}
	}
}
\newcommand{\An}{\operatorname{An}}
\newcommand{\Dens}[1]{\operatorname{Dens}\left(#1\right)}
\newcommand{\swap}{\operatorname{swap}}
\newcommand{\id}[1][] {
	\ifthenelse{\isempty{#1}} {
		\operatorname{id}
	} {
		\operatorname{id}_{#1}
	}
}
\newcommand{\Hom}[3][] {
	\ifthenelse{\isempty{#1}} {
		\operatorname{Hom}\left(#2,#3\right)
	} {
		\operatorname{Hom}_{#1}\left(#2,#3\right)
	}
}
\newcommand{\End}[2][] {
	\ifthenelse{\isempty{#1}} {
		\operatorname{End}\left(#2\right)
	} {
		\operatorname{End}_{#1}\left(#2\right)
	}
}
\newcommand{\Stab}[2][] {
	\ifthenelse{\isempty{#1}} {
		\operatorname{Stab}\left(#2\right)
	} {
		\operatorname{Stab}_{#1}\left(#2\right)
	}
}
\newcommand{\Sch}[2][] {
	\ifthenelse{\isempty{#1}} {
		\mathcal{S}\left(#2\right)
	} {
		\mathcal{S}_{#1}\left(#2\right)
	}
}
\newcommand{\Dist}[2][] {
	\ifthenelse{\isempty{#1}} {
		\mathcal{S}^{*}\left(#2\right)
	} {
		\mathcal{S}_{#1}^{*}\left(#2\right)
	}
}
\newcommand{\supp}[1][] { 
	\ifthenelse{\isempty{#1}} {
		\operatorname{supp}
	} {
		\operatorname{supp}\left(#1\right)
	}
}
\newcommand{\genfun}[2][] {
	\ifthenelse{\isempty{#1}} {
		C^{-\infty}\left(#2\right)
	} {
		C^{-\infty}_{#1}\left(#2\right)
	}
}
\newcommand{\Sm}[2][] {
	\ifthenelse{\isempty{#1}} {
		C^{\infty}\left(#2\right)
	} {
		C_{#1}^{\infty}\left(#2\right)
	}
}
\newcommand{\SmM}[2][] {
	\ifthenelse{\isempty{#1}} {
		\mu^{\infty}\left(#2\right)
	} {
		\mu_{#1}^{\infty}\left(#2\right)
	}
}
\newcommand{\fourier}[1][] {
	\ifthenelse{\isempty{#1}} {
		\mathcal{F}
	}{
		\mathcal{F}\left(#1\right)
	}
}
\newcommand{\WF}[2][] {
	\ifthenelse{\isempty{#1}} {
		\operatorname{WF}\left(#2\right)
	} {
		\operatorname{WF}_{#1}\left(#2\right)
	}
}
\newcommand{\pr}{\operatorname{pr}}
\newcommand{\sm}[1]{{#1}^{\operatorname{sm}}}
\newcommand{\Fr}{\operatorname{Fr}}
\newcommand{\adj}[1][] {
	\ifthenelse{\isempty{#1}} {
		\operatorname{adj}
	} {
		\operatorname{adj}\left(#1\right)
	}
}

\newcommand{\normal}[2][]{
	\ifthenelse{\isempty{#1}} {
		\mathcal{N}_{#2}
	} {
		\mathcal{N}^{#1}_{#2}
	}
}

\begin{document} 
\setlength{\parindent}{0cm}

\title{
	$\Pgl{2}$ is Multiplicity-Free as a $PGL_2\times PGL_2$-Variety \\
	\large MSc thesis completed at Weizmann Institute of Science under the guidance of Prof. Dmitry Gourevitch
}
\author{Shai Keidar}
\date{December 2019}

\begin{titlepage}
\maketitle
\blfootnote{This research was supported by the ERC, StG grant number 637912, grant numbers 249/17}
\begin{abstract}
\setlength{\parindent}{0cm}
Let $F$ be a non-Archimedean local field. Let $G$ be an algebraic group over $F$. A $G$-variety $X$ defined over $F$ is said to be multiplicity-free if for any admissible irreducible representation $\pi$ of $G\left(F\right)$ the following takes place:
\[
\dim\Hom[G\left(F\right)]{\Sch{X\left(F\right)}}{\pi} \le 1
\]
where $\Sch{X\left(F\right)}=\Sm[c]{X\left(F\right)}$ is the space of Schwartz functions on $X\left(F\right)$.
In this thesis we prove that $\Pgl[F]{2}$ is multiplicity-free as a $\PGL[F]{2}\times\PGL[F]{2}$-variety.
\end{abstract}
\end{titlepage}

\tableofcontents
\newpage

\section{Introduction}
\begingroup
\renewcommand\thethm{\Alph{thm}}
This thesis deals with relative representation theory and specifically with non-commutative harmonic analysis on a certain non-homogeneous spherical variety. 
The field of relative representation theory is a prosperous realm for research. A lot of research on symmetric varieties took place in the 1980s to the 2000s (see for example \cite{sym5}, \cite{sym3}, \cite{sym4}, \cite{sym6}, \cite{sym7}, \cite{sym1}, \cite{sym2}, \cite{sym8}) and on toric varieties in the 1990s and 2000s (see \cite{tor1}, \cite{tor3}, \cite{tor2}). Nowadays a main point of interest is the study of spherical varieties, which generlizes both fields (see \cite{sph5}, \cite{sph3}, \cite{sph2}, \cite{sph11}, \cite{sph4}, \cite{sph1}, \cite{sph6}, \cite{sph10}, \cite{sph8}).

\subsection{Spherical Varieties and Multiplicity-Free Varieties}
Let $F$ be a non-Archimedean local field and $\overline{F}$ its algebraic closure. Let $G$ be an algebraic group over $F$ and $B$ a Borel subgroup of $G$.
A $G$-variety $X$ is said to be spherical if it has an open $B$-orbit. A homogenous $G$-variety $X$ is spherical if and only if for any irreducible algebraic representation $\pi$ of $G\left(\overline{F}\right)$ and any $G$-equivariant line bundle $\mathcal{L}$:
\[
\dim\Hom[G\left(\overline{F}\right)]{H^{0}\left(X\left(\overline{F}\right), \mathcal{L}_{\overline{F}}\right)}{\pi} \le 1
\]
in the case $X$ is quasiaffine then this condition is equivalent to the following: For any algebraic representation $\pi$ of $G\left(\overline{F}\right)$:
\[\dim\Hom[G\left(\overline{F}\right)]{\overline{F}\left[X\right]}{\pi} \le 1\]
For a proof see e.g. \cite{sph2} - Theorem 25.1.

A $G$-variety $X$ defined over $F$ is said to be multiplicity-free if for any admissible irreducible representation $\pi$ of $G\left(F\right)$ the following takes place:
\[
\dim\Hom[G\left(F\right)]{\Sch{X\left(F\right)}}{\pi} \le 1
\]
where $\Sch{X\left(F\right)}=\Sm[c]{X\left(F\right)}$ is the space of Schwartz functions on $X\left(F\right)$.
We will also say that $X$ is weakly multiplicity-free if for any admissible irreducible representation $\pi$ of $G\left (F\right )$ the following holds:
\[
\dim\Hom[G\left(F\right)]{\Sch{X\left(F\right)}}{\pi} \cdot
\dim\Hom[G\left(F\right)]{\Sch{X\left(F\right)}}{\widetilde{\pi}}
\le 1
\]
where $\widetilde{\pi}$ is the contragradeint representation of $\pi$.

Multiplicity-free homogenous varieties and the connection between them and homogenous spherical varieties have been the subject to a lot of research. It is known \cite{sphFinMul1},\cite{sphFinMul2} that a real homogenous spherical variety is of finite-multiplicity. The same is conjectured for the non-Archimedean case, and proven in many cases (\cite{pAdicFinMul1}, \cite{pAdicFinMul2}).
A homogenous variety $X=\sfrac{G}{H}$ is multiplicity-free if and only if $\left(G,H\right)$ is a Gelfand pair (see e.g. \cite{AG2} for preliminaries on the notion of Gelfand Pairs).

A homogenous spherical $G$-variety $X$ admits a natural wonderful compactification: A projective $G$-spherical variety $\overline{X}$ containing $X$ as an open dense subvariety such that the closure of each orbit is smooth (\cite{wondComp1}, \cite{wondComp3}, \cite{wondComp2}). We are interested in the wonderful compactification of $G$ as a $G\times G$-variety.

\subsection{Our Contribution}
We look at the wonderful compactificaion of $G=\PGL{2}$ which is $\overline{G}=\Pgl{2}$ with the $\PGL{2}\times\PGL{2}$-action given by left and right multiplication. We prove the following theorem:
\begin{thm}
\label{mainThm}
For a non-Archimedean local field $F$, $\Pgl[F]{2}$ is multiplicity-free as a $\PGL[F]{2}\times \PGL[F]{2}$-variety.
\end{thm}
We hope this thesis will yield more study on harmonic analysis of non-homogenous spherical varieties.

\subsection{Structure of The Proof}
In \cref{weakIFFMulFree} we will prove, using representation theory of $\GL{2}$, that proving \cref{mainThm} is equivalent to proving the following theorem:
\begin{thm}
\label{almostMainThm}
For a non-Archimedean local field $F$, $\Pgl[F]{2}$ is weakly multiplicity-free as a $\PGL[F]{2}\times \PGL[F]{2}$-variety.
\end{thm}

Denote by $X=\Pgl[F]{2}$ and $G=\PGL[F]{2}$. We will use Gelfand-Kazhdan criterion (\cref{GK}) in order to prove \cref{almostMainThm}. To do that it is enough to prove:
\begin{thm}
\label{C}
$\Dist{\XX}^{\GG} \subseteq \Dist{\XX}^{\swap}$ where $\GG$ acts on $\XX$ diagonally.
\end{thm}
\endgroup
\setcounter{thm}{\thetmp}  

or equivalently (\cref{tildeIFF}): 
\[\Dist{\XX}^{\GGt,\chi}=0\]
where $\GGt=\GG\times \ZZ[2]$, $\ZZ[2]$ acts on $\XX$ by $\swap$ and $\chi$ is the character on $\GGt$ which is trivial on $\GG$ but not on $\ZZ[2]$.

We will study the geometry of $\XX$ and try to prove that each $\GGt$-orbit $Y$ in $\XX$ satisfies
\[
\Dist{Y}^{\GGt,\chi} = 0
\]
mainly using Bernstein-Gelfand-Kazhdan-Zelevinsky criterion (\cref{BGKZ}).
This will work for most orbits and will show that is suffices to prove the following:
\[\Dist{C}^{\GGt,\chi}=0\]
where $C$ is a locally closed $\GGt$-equivariant subset of $\XX$ consisting of two orbits $P$ and $W$, s.t. $P$ is also a $\GG$-orbit, $W$ contains two disjoint $\GG$-orbits, and $P$ is contained in the closure of $W$. To prove this we used a variant of the cross method, suggested by Shachar Carmeli: The $\GG$-variety $C$ acts a lot like the $\Gm[F]$-variety $\{xy=0\}$ in $\Aff[F]{2}$, i.e. the cross. We will use the known theorem
\begin{thm}
\label{realCross}
\[\Dist{{\{xy=0\}}}^{\Gm[F]}\subseteq \Dist{{\{xy=0\}}}^{\swap}\]
or equivalently
\[\Dist{{\{xy=0\}}}^{\Gmt[F],\chi}=0\]
where $\Gmt[F]=\Gm[F]\rtimes \ZZ[2]$ and $\chi$ is the character which is trivial on $\Gm[F]$ and non-trivial on $\ZZ[2]$. 
\end{thm}
For a proof see e.g. \cite{AG2}, proposition 3.3.2.
 
We will find an open $\GGt$-equivariant open subvariety $U$ of $\XX$ containing $C$ as a closed subvariety, and a restriction map 
\[ \Dist{U}^{\GGt,\chi} \to \Dist{\Aff[F]{2}}^{\Gmt[F],\chi}\]
that sends $\Dist{C}^{\GGt,\chi}$ injectively into $\Dist{\{xy=0\}}^{\Gmt[F],\chi}$, thus finishing the proof.

\subsection*{Organization of The Thesis}
In \cref{secPerliminaries} we will give preliminary background on a few subjects, including distribution theory on $l$-spaces and on analytic spaces, representation theory of $l$-group and of $\GL{n}$ in particular and the wonderful compactification of semi simple groups of adjoint type.

\cref{weakIFFMulFree} is devoted to proving that \cref{mainThm} is equivalent to \cref{almostMainThm}, using the representation theory of $\GL{n}$.

In \cref{redToCross} we will look at the Gelfand-Kazhdan criterion and will reduce to proving a simpler theorem - \cref{C}. We will study the geometry of $\XX$, identify a subset $C\subseteq U\subseteq X$, where $C$ is close in $U$ and behaves like the cross and $U$ is open in $\XX$. Lastly we will prove that it is enough to prove the null-multiplicity only for $C$ (\cref{D}).

In \cref{crossMethod} we will introduce the variant of the cross method we use. We will find a restriction map from $\GGt,\chi$-invariant generalized functions on $U$ to $\Gmt,\chi$-invariant generalized functions on the plane in which generalized functions on $C$ are mapped to generalized functions on the cross $\{xy=0\}$. We will show this map is injective and handle the difference between generalized functions and distributions.

\subsection*{Acknowledgments}
First of all I would like to thank my advisor Dmitry Gourevitch for guiding me through this project with endless patience, and showing me his approach to mathematics.

I would like to thank my parents and my family for my education and for supporting me in every life decision I made.

I also wish to thank Shachar Carmeli for suggesting the cross method, and for all the useful conversations and suggestions.

I wish to thank Yotam Hendel and Itay Glazer for making my working hours in Weizmann Institute enjoyable and for providing me with coffee. I could not have completed this work without it.

\section{Preliminaries}
\label{secPerliminaries}
Let $F$ be a non-Archimedean local field and $\overline{F}$ its algebraic closure.

\subsection{$l$-Spaces, Analytic Varieties And Distribution}
\begin{dfn}
A topological spaces $X$ is called an $l$-space if it is Hausdorff, locally compact and totally disconnected.

A topological group $G$ is called an $l$-group if it is also an $l$-space as a topological space.
\end{dfn}

\begin{lem}
Let $X$ be an $l$-space and $Y\subseteq X$ a locally-closed subset. Then $Y$ is an $l$-space.
\end{lem}

\begin{lem} (See \cite{BZ1} 6.5)
Let $G$ be an $l$-group and $H$ a closed subgroup, then $\sfrac{G}{H}$ with the quotient topology is an $l$-space.
\end{lem}

\begin{dfn}
Let $X$ be an $l$-space. Define the space $\Sm{X}$ of smooth functions on $X$ to be the space of locally constant functions $X\to \bb{C}$. Define the space $\Sch{X}=\Sm[c]{X}$ of Schwartz function on $X$ to be the space of smooth, compactly supported function $X\to \mathbb{C}$. We equip these spaces with the discrete topology. \\
We may look at the dual space $\Dist{X}$ - the space of distributions on $X$. We equip this space with the weak topology - the topology generated by 
\[U_{\eps,f}:=\left\{\xi\in\Dist{X}\,|\, \left|\langle \xi, f\rangle\right|< \eps\right\}\]
where $\eps>0$ and $f\in \Sch{X}$.
\end{dfn}

\begin{lem}
\label{exactSeq}
Let $X$ be an $l$-space, $U\subseteq X$ open and $Z=X\setminus U$. We have maps $\Sch{X}\to\Sch{Z}$ given by restriction, and $\Sch{U}\to \Sch{X}$ given by continuation by zeros. The sequence
\[0\to\Sch{U}\to\Sch{X}\to\Sch{Z}\to 0\]
is exact.
\end{lem}

\begin{dfn}
An $F$-analytic manifold of dimension $n$ is a ringed space $\left(M,O\right)$ which is locally isomorphic to $\left(\cal{O}_{F}^{n}, \An\right)$ where
\[
\An\left(U\right) = 
\left\{
	f:U \to F\
	\,\bigg|\,
	\forall x\in U, \exists r>0 \text{ s.t. }
		f\big|_{B_r\left(x\right)} \left(y\right) = 
		\sum_{
			\overrightarrow{\al}\in \bb{Z}_{\ge 0}^n
		} a_{\overrightarrow{\al}}
		\left(x-y\right)^{\overrightarrow{\al}}
\right\}
\]
\end{dfn}

\begin{lem}
Let $M$ be an $F$-analytic manifold. Then $M$ is an $l$-space.
\end{lem}

\begin{lem}
Let $X$ be a smooth algebraic variety defined over $F$. Then $X\left(F\right)$ has a natural $F$-analytic structure.
\end{lem}

\begin{dfn}
Let $M$ be an $F$-analytic manifold, and $p:E\to M$  a complex or real vector bundle over $M$. We define the space $\Sm{M,E}$ of smooth sections to be the space of locally constant sections of $E$. We define the space $\Sch{M,E}$ of Schwartz sections to be the space of smooth, compactly supported sections of $E$. We equip these spaces with the discrete topology. \\
We may look at the dual space $\Dist{M,E}$ - the space of distributional $E$-section on $M$. We equip this space with the weak topology - the topology generated by 
\[U_{\eps,f}:=\left\{\xi\in\Dist{M,E}\,|\, \left|\langle \xi, f\rangle\right|< \eps\right\}\]
where $\eps>0$ and $f\in \Sch{M,E}$.

We will be interested in a few bundles over $M$: the constant bundle $\bb{C}_M$, the bundle $\det\left(M\right):=\Lambda^{\text{top}} \left(T^* M\right )$ of $F$-valued top differential forms, and the density bundle $\Dens{M}:=\left|\det\left(M\right)\right|$. Note that $\Sch{M,\bb{C}_M} = \Sch{M}$, $\Dist{M,\bb{C}_{M}}=\Dist{M}$. The space of smooth, compactly supported measures on $M$ is $\SmM[c]{M}:=\Sch{M,\Dens{M}}$
We define the space of generalized $E$-sections as \[\genfun{M,E}:=\Dist{M,E^{*}\otimes \Dens{M}}\]
The space of generalized functions on $M$ is \[\genfun{M}:=\genfun{M,\bb{C}_{M}}=\SmM[c]{M}^{*}\]
\end{dfn}

Note that $\Sm{M,E}$ embeds naturally into $\genfun{M,E}$

\begin{thm}
$\Sm{M,E}$ is dense in $\genfun{M,E}$ (w.r.t. the weak topology).
\end{thm}

\subsection{Wave Front Sets and Push Forward of Generlized Functions}
We wish to know when can we push forward generlized functions, and specifically when can a generalized function $\xi$ on a variety $X$ be restricted to a subvariety $Y\subseteq X$ . Harish-Chandra showed the existence of a push forward in the case of a submersion (see \cite{HC}, \cite{Dima1}). H\"ormander gave sufficient conditions for the existence of push forward for the real case  in \cite{Hor} (theorem 8.2.4) using his notion of wave front sets - the set of directions on which the distribution is not smooth.
This notion of wave-front sets and the criterion was extended to the non-Archimedean case by Heifetz in \cite{wf} (Theorem 2.8). See also \cite{Aiz1}, \cite{Aiz2} for the notion of wave front sets in the non-Archimidean case.

\begin{thm}[Harish-Chandra's submersion principle]
\label{HCS}
Let $p:M\to N$ be a submersion of $F$-analytic manifolds. Let $E$ be a vector bundle over $N$. Then there exists a surjective continuous linear map 
\[p_{*}:\Sch{M, p^{*}E\otimes \Dens{M}}\to \Sch{N,E\otimes \Dens{N}}\]
s.t. for any $f\in \Sch{N,E^{*}}$ and $\mu\in\Sch{M,p^{*}E\otimes \Dens{M}}$ we have
\[
\int_{N} \left\langle f, p_{*}\mu \right\rangle
= \int_{M} \left\langle f\circ p, \mu \right\rangle
\]
Moreover $p_{*}\mu$ is unique in $\Sch{N, E\otimes\Dens{N}}$ w.r.t this property.
\end{thm}

\begin{dfn}
Let $V$ be a finite dimensional $F$-vector space. Define the Fourier transform of distributions on $V$, 
\[\fourier: \Dist{V}\to\genfun{V^{\vee}}\]
as the dual operator of the Fourier transform \[\fourier: \SmM[c]{V^{\vee}} \to \Sch{V}\]
defined by
\[\fourier\mu\left(v\right) = \int_{V^\vee} \chi\left(v\right) \,d\mu \left(\chi\right)\]
where $V^{\vee}$ is the Pontryagin dual of $V$. By choosing $\mu_0\in F^\vee$ we can identify $V^*$ with $V^\vee$ using the isomorphism $\varphi\mapsto \mu_0\circ \varphi$. And by choosing a basis for $V$, we will identify $V^\vee$ with $V$.
\end{dfn}

\begin{claim}
Let $\xi\in\Dist{V}$ be a distribution with compact support. Then $\fourier[\xi]\in\Sm{V}$.
\end{claim}

\begin{dfn}
Let $V$ be a vector space over $F$.
\begin{enumerate}
\item Let $v\in V$ and $f\in\Sm{V}$. We say that $f$ vanishes asymptotically along $v$ if there exists an open neighborhood $U\subseteq V$ of $v$ and $\rho\in\Sm[c]{U}$ such that $\left(p^* \rho\right)\cdot \left(m^* f\right) \in \Sch{U\times F}$ where $m:V\times F\to V$ is given by $m\left(v,\la\right) = \la v$ and $p:V\times F\to F$ is the projection.

\item Let $\xi\in\Dist{V}$. We say that $\xi$ is smooth at $\left(x,w\right)\in V\times V^*$ if there exists $\rho\in\Sm[c]{V}$ such that $\rho\left(x\right)=1$ and $\fourier[\rho\xi]$ vanishes asymptotically along $w$.

\item Let $\xi\in\Dist{V}$. Define its wave front set by 
\[
\WF{\xi}:=
\left\{
	\left(x,w\right)\in V\times V^*		\,|\,
	\xi \text{ is not smooth at } \left(x,w\right)
\right\}
\]
For a point $x\in V$ let $\WF[x]{\xi}:=\WF{\xi}\cap \left\{x\right\}\times V^*$.
\end{enumerate} 
\end{dfn}

\begin{dfn}
Let $\nu:X\to Y$ be a morphism of $F$-manifolds and $\Ga\subseteq T^* Y$. We define its pullback by 
\[\nu^* \Ga:= \left\{
	\left(x, \eta\right) \in T^{*}X
	\,|\,
	\exists \eta'\in T^{*}_{f\left (x\right )}Y: \left(f\left(x\right), \eta'\right)\in \Ga,
	d^*_{f\left(x\right)}\nu \left(\eta'\right)= \eta
\right\}
\]
\end{dfn}

\begin{thm}
Let $\nu:V\to V$ be a diffeomorphism of the $F$-vector space $V$ and $\xi\in\Dist{V}$. Then \[\WF{\nu^*\xi}=\nu^*\WF{\xi}\]
\end{thm}

\begin{cor}
The definition of wave front set extends to generalized sections of vector bundles over manifolds.
\end{cor}

\begin{dfn}
Let $\Ga\subseteq T^{*}M$ be a closed subset. Define \[\genfun[\Ga]{M,E}:=\left\{ \xi\in\genfun{M,E}\,|\,\WF{\xi}\subseteq\Ga \right\}\]
\end{dfn}

We equip $\genfun[\Ga]{M,E}$ with a topology. In order to do this it is enough to define the topology for the case where $M=V$ is a vector space:
\begin{dfn}
Let $V$ be an $F$-vector space. Define a topology on $\genfun[\Ga]{V}$ by $\xi_{n}\to \xi$ if $\xi_{n}\to\xi$ weakly in $\genfun{V}$ and for any $v\in V$ $\exists \eps >0$ and $\rho\in\Sm[c]{B_{\eps}\left(v\right)}$ such that $\forall \vp\in V^{*}$:
\[m^{*}\fourier[\rho\xi_n]\big|_{B_{\eps}\left(\vp\right)\times F}
 \to m^{*}\fourier[\rho\xi]\big|_{B_{\eps}\left(\vp\right)\times F}\]
\end{dfn}

As smooth function have trivial wave front sets, $\Sm{M}$ embeds into $\genfun[\Ga]{M}$.
\begin{thm}
$\Sm{M}$ is dense in $\genfun[\Ga]{M}$ for any closed $\Ga\subseteq T^{*}M$.
\end{thm}

\begin{dfn}
Let $\Lambda\subseteq T^{*}M$ be a subset. Define 
\[
	\genfun[\Lambda]{M}
	:= \bigcup_{ 
		\substack{
			\Ga\subseteq T^{*}M \text{ closed} \\
			\Ga\subseteq \Lambda\cup M\times\left\{0\right\}
		}
	} 
	\genfun[\Ga]{M}
\]
Equip it with the colimit topology. As $\Sm{M}\subseteq \genfun[\Ga]{M}$ is dense for any $\Ga$ it is also dense in $\genfun[\Lambda]{M}$.
\end{dfn}

\begin{dfn}
Let $\nu:M\to N$ be a map of $F$-manifolds. Define 
\[S_{\nu}:=\left\{
	\left( \nu\left(x\right), w\right)\in T^{*}N
	\, | \,
	x\in M, d_{x}^{*}\nu\left( w \right) = 0
\right\}\]
\end{dfn}

\begin{exmd}
If $\iota: M\hookrightarrow N$ is an inclusion of manifolds, then $S_{\iota} = \normal[N]{M}$
\end{exmd}

\begin{thm}[\cite{wf}]
Let $\nu: M\to N$ be a map of $F$-manifolds, and $E$ a vector bundle over $N$. Let $\Ga\subseteq T^* N$ be a closed subset such that $\Ga\cap S_{\nu} \subseteq N\times\left \{0\right\}$. Then the pullback map of smooth functions
\[\nu^{*}:\Sm{N, E}\to\Sm{M, \nu^{*}E}\]
has a unique continuous extension 
\[\nu^{*}:\genfun[\Ga]{N,E}\to\genfun[\nu^{*}\left(\Ga\right)]{M,\nu^{*}E}\]
Moreover, for any $\xi\in\genfun[\Gamma]{N,E}$ we have $\supp[\nu^{*}\xi]\subseteq \nu^{-1}\left(\supp[\xi]\right)$.
\end{thm}

\begin{cor}
Let $\nu: M\to N$ be a map of $F$-manifolds, and $E$ a vector bundle over $N$.
Then the pull back map of smooth functions
\[\nu^{*}:\Sm{N, E}\to\Sm{M, \nu^{*}E}\]
has a unique continuous extension 
\[\nu^{*}:\genfun[S_{\nu}^{c}]{N,E}\to\genfun{M,\nu^{*}E}\]
Moreover, for any $\xi\in\genfun[S_{\nu}^{c}]{N,E}$ we have $\supp[\nu^{*}\xi]\subseteq \nu^{-1}\left(\supp[\xi]\right)$.
\end{cor}

\begin{cor}
\label{rest-crit}
Let $M\subseteq N$ be a submanifold, and $E$ a vector bundle over $N$.
Then we can define a restiction map 
\[\genfun[\left({\normal[N]{M}}\right)^{c}]{N,E}\to\genfun{M,E\big|_{M}}\]
as the pullback by the inclusion map, and
\[\supp[\xi\big|_{M}] \subseteq \supp[\xi] \cap M\]
for any $\xi\in\genfun[\left({\normal[N]{M}}\right)^{c}]{M}$.
\end{cor}

\subsection{Representation Theory of $l$-Groups}
Let $G$ be an $l$-group.
\begin{dfn}
A morphism $\pi:G\to \GL[V]{}$ for some complex vector space $V$ is called a representation of $V$.
\end{dfn}
\begin{dfn}
Given a representation $\pi:G\to \GL[V]{}$, a vector $v\in V$ is called smooth if its stabilizer $G_{v}$ is open in $G$. The subspace of smooth vectors $\sm{V}$ is a subrepresentation denoted by $\sm{\pi}:G\to \GL[\sm{V}]{}$.
$\pi$ is smooth if $V=\sm{V}$.
\end{dfn}
\begin{dfn}
Given a representation $\pi:G\to V$ we define the dual representation $\pi^{*}:G\to \GL[V^{*}]{}$ given by
\[
\left\langle \pi^{*}\left(g\right)\varphi, v\right\rangle
= \left\langle \varphi, \pi\left(g^{-1}\right)v\right\rangle
\]
The contragredient representation of $\pi$ is $\widetilde{\pi}:=\sm{\left(\pi^{*}\right)}$.
\end{dfn}
\begin{dfn}
A representation $\pi:G\to\GL[V]{}$ is admissible if for any compact subgroup $K\subseteq G$, the space of $K$-invariants $V^{K}$ is finite-dimensional.
\end{dfn}

\subsection{Representations of $\GL[F]{n}$}
\begin{thm}[Bernstein-Zelevinsky, \cite{BZ1}-3.25]
Let $\pi$ be a smooth irreducible representation of $\GL[F]{n}$, then $\pi$ is admissible.
\end{thm}
\begin{thm}[Gelfand-Kazhdan, \cite{GK}, \cite{BZ1}-7.3]
\label{repOfGLn}
Let $s_{n}\in \GL{n}$ be the matrix given by
\[\left(s_{n}\right)_{i,j}=(-1)^{i}\cdot \delta_{i,n+1-j}\]
Define an automorphism $s$ of $\GL{n}$ given by 
\[g^{s}=s_{n}g^{-T}s_{n}^{-1}\]

Let $\pi$ be an irreducible admissible representation of $\GL{n}$. Let $\pi^{s}$ be the representation $\pi$ with $\GL{n}$-action twisted by the isomorphism $s$. Then $\pi^{s}\cong \widetilde{\pi}$.
\end{thm}

\subsection{Group Actions and Distribution Theory}
\begin{dfn} 
Let $G$ be an $l$-group acting on an $l$-space $X$. We say that $X$ is multiplicity-free if for any admissible irreducible representation $\pi$ of $G$ we have 
\[  
	\dim\Hom[G]{\Sch{X}}{\pi} \le 1
\]
We say that is weakly multiplicity-free if for any admissible irreducible representation $\pi$ of $G$ we have 
\[  
	\dim\Hom[G]{\Sch{X}}{\pi}\cdot 
	\dim\Hom[G]{\Sch{X}}{\widetilde{\pi}}
	\le 1
\]
\end{dfn}

Gelfand and Kazhdan gave a criterion for a homogenous variety $X=G/H$ to be weakly multiplicity-free \cite{GK}. This criterion was generalized to the non-homogenous case in \cite{AG3} (Theorem 2.4.1):
\begin{criterion}[Gelfand-Kazhdan]
\label{GK}
Let $G$ be an $l$-group acting on an $l$-space $X$. Let $\tht$ be an involution of $G$ and $X^{\tht}$ be the space $X$ with $G$-action twisted by $\tht$. Let $\al:X\to X^{\tht}$ be a $G$-isomorphism satisfying $\alpha\circ\alpha=\id$. Assume that $\Dist{X\times X}^G \subseteq \Dist{X\times X}^{\swap\circ\Del\al}$. Then X is weakly multiplicity-free.
\end{criterion}

Another important criterion porven by Bernstein and Zelevinsky in \cite{BZ1} is the following:
\begin{criterion}[Bernstein-Gelfand-Kazhdan-Zelevinsky]
\label{BGKZ}
Let $G$ be an algebraic group defined over $F$. Let $X$ be a $G$-variety defined over $F$. Let $\cal{F}$ be a $G\left(F\right)$-equivariant sheaf over $X\left(F\right)$. Assume that for any $x\in X\left(F\right)$
\[
\left(
	\cal{F}_x \otimes 
	\Del_{G\left(F\right)}\big|_{G\left(F\right)_x} \otimes 
	\Del_{G\left(F\right)_x}^{-1}
\right) ^ {G\left(F\right)_x} = 0
\]
where $\Del_H$ is the modular character of the group $H$ and $G\left(F\right)_x$ is the stabilizer of $x$ in $G\left(F\right)$. Then 
\[\Dist{X\left(F\right),\cal{F}}^{G\left(F\right)} = 0\]
\end{criterion}

\begin{thm}[Frobenius Descent, \cite{Ber1}-1.5]
\label{frobDesc}
Let $G$ be an $F$-analytic group acting on an $F$-analytic manifold $M$. Let $Z$ be a homogenous $F$-analytic $G$-manifold. Let $\phi: M\to Z$ be a $G$-equivariant map. Let $z\in Z$ be a point, $M_{z}=\phi^{-1}\left(z\right)$ the primage of $z$ in $M$ and $G_{z}$ the stabilizer of $z$ in $G$. Let $\Delta_{G}$ and $\Delta_{G_{z}}$ be the unimodular characters of $G$ and $G_{z}$ respectively. Let $\chi$ be a character on $G$. Let $E$ be a $G$-equivariant vector bundle over $M$. Then there exists a canonical isomorphism
\[
\Fr: \genfun{M_{z}, E\big|_{M_{z}}\otimes \Delta_{G}\big|_{G_{z}}\otimes \Delta_{G_{z}}^{-1} }^{G_z,\chi|_{G_{z}}\cdot \Delta_{G}\big|_{G_{z}} \cdot \Delta_{G_{z}}^{-1}} \xrightarrow{\sim}
\genfun{M,E}^{G,\chi}
\]
That on smooth sections is given by
\[
\Fr\left(f\right)(x) = \chi\left(g_{x}\right)^{-1} f\left(g\act x\right)
\]
For $f\in \Sm{M_{z}, E\big|_{M_{z}}\otimes \Delta_{G}\big|_{G_{z}}\otimes \Delta_{G_{z}}^{-1}}^{G_{z}, \chi\cdot \Delta_{G}\big|_{G_{z}}\cdot \Delta_{G_{z}}^{-1}}$, and where $g_{x}\in G$ is such that $g_{x}\act x\in M_{z}$ for $x\in M$.
\end{thm}

\begin{thm}[Aizenbud, \cite{Aiz1}-4.15]
\label{action-WF}
Let $G$ be an algebraic group over $F$ acting on a variety $X$ over $F$. For $x\in X\left (F \right )$ denote the action map 
$ a_{x}: G\left (F \right ) \to X \left (F \right )$
given by 
$ a_{x} \left ( g \right ) = g \act x $.
Write its differential at $e$ as 
$ da_{x} : \frak{g} \left ( F \right ) \to T_{x}X\left ( F \right )$.
Let $\xi\in\Dist{X\left(F\right)}^{G\left(F\right)}$.
Then 
$
	\WF[x]{\xi} 
	\subseteq \left ( 
		da_{x} \left (
			\frak{g} \left (F \right )
		\right )
	\right )^{\perp}
$.
\end{thm}

\subsection{The Wonderful Compactification of $PGL_{n}\left(F\right)$}
\label{secWondComp}
In \cite{onTheWondComp} we are given a construction of the wonderful compactification of a semisimple group $G$ of adjoint type:
Let $\pi: \widetilde{G}\to G$ be the algebraic universal cover of $G$, $\al_{1},\dots,\al_{l}$ be the simple roots $\widetilde{G}$, with relation to some torus and Borel subgroups $\widetilde{T}\subseteq \widetilde{B}\subseteq \widetilde{G}$. Fix an irreducible representation $V$ of $\widetilde{G}$ with regular highest weight $\la$ and a basis $v_{0},\ldots,v_{n}$ of $\widetilde{T}$-weight vectors of $V$ with the following properties:
\begin{enumerate}
\item $v_{0}$ has weight $\la$
\item For $i=0,\ldots ,l$: $v_{i}$ has weight $\lambda-\al_{i}$
\item Let $\lambda_{i}$ be the weight of $v_{i}$. Then if $\lambda_{i} < \lambda_{j}$ then $i > j$
\end{enumerate}
$G\times G$ acts on $\bb{P}\left(\End{V}\right)$. We have a $G\times G$-embedding $\psi:G\hookrightarrow \bb{P}\left(\End{V}\right)$ given by $\psi\left(g\right) = \left[g\right]$. Then we define the wonderful compactification of $G$ to be $\overline{G}:=\overline{\psi\left(G\right)}$ in $\bb{P}\left(\End{V}\right)$.

For our example let $G=PGL_{n}$, $\widetilde{G}=SL_{n}$, $\widetilde{B}$ - the subgroup of upper-triangular matrices and $\widetilde{T}$ - the subgroup of diagonal matrices. The simple roots are $\left\{\alpha_{1},\dots,\alpha_{n-1}\right\}$ where $\alpha_i=\epsilon_{n}-\epsilon_{i}$ and $\epsilon_{i}:\widetilde{T}\to \bb{G}_{m}$ is given by 
$\epsilon_{i}\left(\begin{array}{cccc}
x_{1} &		  &		   &		\\
	  & x_{2} &		   &		\\
	  &		  & \ddots &		\\
	  &		  &		   & x_{n}
\end{array}\right) = x_{i}$.
Let $V = F^{n}$ be the standard representation of $\widetilde{G}$. It has a highest weight $\lambda = \epsilon_{n}$ and a basis of weights: $v_{0}= e_{n}$ with weight $\lambda=\epsilon_{n}$, $v_{i}=e_{i}$ with weight $\lambda_{i}=\epsilon_{i}=\lambda-\alpha_{i}$.
So $V$ is a representation with the required properties. $\End{V}=\gl{n}$ and $G\times G$ acts on $\Pgl{n}$ by
\[\left(\left[g\right],\left[h\right]\right)\act \left[x\right] = 
\left[gxh^{-1}\right]\]
As $G\subseteq \Pgl{n}$ is dense we find that $\overline{G} = \Pgl{n}$.

\section{Notations}
We prove that the wonderful compactification of $\PGL{2}$ is multiplicity free over $F$. Let $G=\PGL[F]{2}$ and $X=\Pgl[F]{2}$ - its wonderful compactification equipped with the $\GG$ action given by matrix multiplication on both sides.
For most of this thesis we will denote elements of $G$ as $g$ (opposed to $\left[g\right]$) and elements of $X$ as $x$.

We will denote by $B\subseteq G$ the Borel subgroup of upper-triangular matrices and by $T\subseteq B$ the maximal torus of diagonal matrices. We will also denote by 
$Q:=\left\{
	\left[\begin{array}{cc}
		0 & *	\\
		* & 0
\end{array}\right]\right\}\subseteq G$, 
and notice that $Q^2 = T$.

We will name the following matrices:
\[
\begin{aligned}
I       & = & \left[\begin{array}{cc} 1&0\\0&1 \end{array}\right]	&\ \ \ \ &
E_{1,1} & = & \left[\begin{array}{cc} 1&0\\0&0 \end{array}\right]	&\ \ \ \ &
E_{1,2} & = & \left[\begin{array}{cc} 0&1\\0&0 \end{array}\right]	\\
E_{2,1} & = & \left[\begin{array}{cc} 0&0\\1&0 \end{array}\right]	&\ \ \ \ &
E_{2,2} & = & \left[\begin{array}{cc} 0&0\\0&1 \end{array}\right]	&\ \ \ \ &
S       & = & \left[\begin{array}{cc} 0&1\\1&0 \end{array}\right]
\end{aligned}\]

\section{Reduction to Weakly Multiplicity-Free}
\label{weakIFFMulFree}
We wish to prove \cref{mainThm}. The Gelfand-Kazhdan criterion (\cref{GK}) is a powerful tool for proving that varieties are weakly multiplicity free. In the following sections we will use the Gelfand-Kazhdan criterion in order to prove \cref{almostMainThm}.

This section is devoted to proving that both theorems are equivalent. We will prove the more general statement:
\begin{thm}
\label{AandBareEquivalent}
$\Pgl[F]{n}$ is multiplicity-free as a $\PGL[F]{n}\times\PGL[F]{n}$-variety if and only if it is weakly multiplicity-free.
\end{thm}
To prove this theorem, remember the isomorphism $s$ of $\GL[F]{n}$ from \cref{repOfGLn}. It is easy to see that for $\lambda\in \Gm[F]$ this isomorphism satisfies $\left(\lambda \cdot g\right)^{s}=\lambda \cdot g^{s}$. Thus we can think of this isomorphism as an isomorphism of $\PGL[F]{n}$ and we have a commutative diagram
\[\xymatrix{
	\GL[F]{n}
		\ar[r]^{s}
		\ar[d]^{\pr}
		&
	\GL[F]{n}
		\ar[d]^{\pr}
		\\
	\PGL[F]{n}
		\ar[r]^{s}
		&
	\PGL[F]{n}
}\]
\begin{lem}
Let $\pi$ be an irreducible admissible representation of $\PGL[F]{n}$. Then there exists a $\PGL[F]{n}$-equivariant isomorphism
\[
\Hom{\Sch{\Pgl[F]{n}}}{\pi}^{s} \xrightarrow{\sim} \Hom{\Sch{\Pgl[F]{n}}}{\tilde{\pi}}
\]
Where for a represenation $V$ of $\PGL[F]{n}$ the notation $V^{s}$ stands for the representation with a $\PGL[F]{n}$-action twisted by $s$.
\end{lem}
\begin{proof}
$\pi$ is an admissible irreducible representation of $\PGL[F]{n}$ and thus also of $\GL[F]{n}$.
By \cref{repOfGLn}, $\pi^{s}\cong \tilde{\pi}$ and whence it is equivalent to prove that we have a $\PGL[F]{n}$-isomorphism:
\[
\phi: \Hom{\Sch{\Pgl[F]{n}}}{\pi}^{s} \xrightarrow{\sim} \Hom{\Sch{\Pgl[F]{n}}}{\pi^{s}}
\]
Let $f\in\Hom{\Sch{\Pgl[F]{n}}}{\pi}^{s}$. Define $\phi f$ in the following way:
\[
\phi f \left(x\right) = f\left(s_{n}\cdot \adj[x^{T}]\cdot s_{n}^{-1}\right)
\]
Where $\adj[x]$ is the adjugate matrix of $x$.
Now:
\[\begin{aligned}
	\left(g\act \phi f\right)\left(x\right) 
	& = \pi^{s}\left(g\right) \phi f\left(g^{-1}\act x\right) 											\\
	& = \pi\left(g^{s}\right) f\left(s_{n} \adj[\left(g^{-1} x\right)^{T}] s_{n}^{-1}\right)			\\
	& = \pi\left(g^{s}\right) f\left(s_{n} g^{T} \adj[x^{T}] s_{n}^{-1}\right)							\\	
	& = \pi\left(g^{s}\right) f\left(s_{n} g^{T} s_{n}^{-1} s_{n} \adj[x^{T}] s_{n}^{-1}\right)			\\
	& = \pi\left(g^{s}\right) f\left(\left(g^{s}\right)^{-1} s_{n}^{-1} \adj[x^{T}] s_{n}^{-1}\right)	\\
	& = \left(g^{s}\act f\right)\left(s_{n}^{-1} \adj(x^{T})s_{n}\right)								\\
	& = \phi\left(g^{s}\act f\right)\left(x\right)
\end{aligned}\]
Whence $\phi$ is $\PGL[F]{n}$-equivariant. It is left to prove that $\phi$ is invertible which is easy, and moreover:
\[\phi^{-1}f\left(x\right) = f\left(s_{n}^{T}\cdot \adj[x^{T}]\cdot s_{n}^{-T}\right)\]
\end{proof}

\begin{proof}[Proof of \cref{AandBareEquivalent}]
Assume that $\Pgl[F]{n}$ is weakly multiplicity-free. Let $\pi$ be an admissible irreducible representation of $\PGL[F]{n}$. By the above lemma
\[
\Hom{\Sch{\Pgl[F]{n}}}{\pi}^{s} \cong \Hom{\Sch{\Pgl[F]{n}}}{\widetilde{\pi}}
\]
as $\PGL[F]{n}$-representations. In particular
\[
\begin{aligned}
\dim \left(\Hom{\Sch{\Pgl[F]{n}}}{\pi}^{s}\right)^{\PGL[F]{n}}
& = \dim \Hom{\Sch{\Pgl[F]{n}}}{\widetilde{\pi}}^{\PGL[F]{n}}	\\
& = \dim \Hom[{\PGL[F]{n}}]{\Sch{\Pgl[F]{n}}}{\widetilde{\pi}}
\end{aligned}
\]
The twisted action does not change the invariants, whence:
\[
\begin{aligned}
\dim \left(\Hom{\Sch{\Pgl[F]{n}}}{\pi}^{s}\right)^{\PGL[F]{n}}
& = \dim \Hom{\Sch{\Pgl[F]{n}}}{{\pi}}^{\PGL[F]{n}}	\\
& = \dim \Hom[{\PGL[F]{n}}]{\Sch{\Pgl[F]{n}}}{{\pi}}
\end{aligned}
\]
Thus we got
\[
\dim \Hom[{\PGL[F]{n}}]{\Sch{\Pgl[F]{n}}}{\widetilde{\pi}}=
\dim \Hom[{\PGL[F]{n}}]{\Sch{\Pgl[F]{n}}}{{\pi}}
\]
Since $\Pgl[F]{n}$ is weakly multiplicity-free
\[
\dim \Hom[{\PGL[F]{n}}]{\Sch{\Pgl[F]{n}}}{{\pi}} \cdot
\dim \Hom[{\PGL[F]{n}}]{\Sch{\Pgl[F]{n}}}{\widetilde{\pi}} \le 1
\]
Or equivalently:
\[\left(\dim \Hom[{\PGL[F]{n}}]{\Sch{\Pgl[F]{n}}}{{\pi}}\right)^{2} \le 1\]
Which implies that:
\[\dim \Hom[{\PGL[F]{n}}]{\Sch{\Pgl[F]{n}}}{{\pi}} \le 1\]
i.e. $\Pgl[F]{n}$ is multiplicity-free.
\end{proof}

\section{Reducing to The Cross}
\label{redToCross}
\subsection{Gelfand-Kazhdan Criterion}
We now want to prove \cref{almostMainThm}, which can be done using the Gelfand-Kazhdan criterion (\cref{GK}).
We look at the identity involution $\theta = id$ on $\GG$, and we let $\alpha:X\to X^{\theta}$ be the identity morphism. Using Gelfand-Kazhdan criterion we wish to prove \cref{C}.
Let $\GGt:=\GG \times \ZZ[2]$, and denote the elements of $\ZZ[2]$ by $\left\{1,\epsilon\right\}$.
Let $\chi:\GGt\to\bb{C}^{\times}$ be the character which is trivial on $\GG$ and satisfies $\chi\left(\epsilon\right)=-1$.

\begin{lem}
\label{tildeIFF}
Let $Y\subseteq X\times X$ be a symmetric $\GG$-variety, i.e. $Y^{\swap} = Y$. $\GGt\acts Y$ by $\epsilon\act y = y^{\swap}$.
Then $\Dist{Y}^{\GG}\subseteq \Dist{Y}^{\swap}$ if and only if $\Dist{Y}^{\GGt,\chi} = 0$.
\end{lem}

\begin{proof}
Assume that $\Dist{Y}^{\GG} \subseteq \Dist{Y}^{\swap}$. Let $\xi\in\Dist{Y}^{\GGt,\chi}$, in particular $\xi\in \Dist{Y}^{\GG}\subseteq \Dist{Y}^{\swap}$ and $\xi\in\Dist{Y}^{\ZZ[2], \chi} = \Dist{Y}^{\swap,-1}$. Whence $\xi=0$.

On the other hand, assume that $\Dist{Y}^{\GGt, \chi}=0$. Let $\xi\in\Dist{Y}^{\GG}$. Then $\xi-\xi^{\swap}\in \Dist{Y}^{\GGt,\chi}=0$. Therefore $\xi=\xi^{\swap}$, i.e. $\xi\in\Dist{Y}^{\swap}$.
\end{proof}

\subsection{The Geometry of $\XX$}
In this section we will study the geometry of $\XX$ and its decomposition to $\GGt$-orbits.

Denote by 
\[O=\GG\subseteq \XX\]
the open $\GGt$-subvariety. 
Let 
\[X_{1}:=\XX\setminus O=\left \{
	\left (x,y\right )\in\XX
	\,|\,
	x,y \text{ are not both invertible}
\right \}\]

For a point $\left(x,y\right)\in X_1$ denote its orbit by 
\[O_{x,y}=\GGt\act \left(x,y\right)\]
and its stabilizer by 
\[P_{x,y}=\Stab[\GGt]{x,y}\]
We have $O_{x,y}\cong \sfrac{\GGt}{P_{x,y}}$ whence 
\[\dim O_{x,y}=\dim\left(\GGt\right) - \dim P_{x,y}=6-\dim P_{x,y}\]

\begin{claim} 
\label{stabilizers}
These are the the stabilizers and orbit-dimensions of our points of interest:
\begin{enumerate}
	\item $ P_{E_{1,2},E_{1,2}} = B\times B\times \ZZ[2] \Rightarrow
	   \dim O_{E_{1,2},E_{1,2}} = 2 $
	\item $ P_{E_{1,2},E_{1,1}} = B\times \left(
				T\times \left\{1\right\} \cup 
				Q\times\left\{\epsilon\right\}
			\right) \Rightarrow
	   \dim O_{E_{1,2},E_{1,1}} = 3 $
	\item $ P_{E_{1,2},E_{2,2}} =  \left(
				T\times \left\{1\right\} \cup 
			    Q\times\left\{\epsilon\right\}
			\right) \times B \Rightarrow
	   \dim O_{E_{1,2},E_{2,2}} = 3 $
	\item $ P_{E_{1,2},E_{2,1}} = T\times T\times \left\{1\right\}
				\cup Q\times Q\times \left\{\epsilon\right\}
				\Rightarrow
	   \dim O_{E_{1,2},E_{2,1}} = 4 $
	\item $ P_{I      ,E_{1,2}} = \Delta B \Rightarrow
	   \dim O_{I      ,E_{1,2}} = 4 $
	\item $ P_{I      ,E_{1,1}} = \Delta T \Rightarrow
	   \dim O_{I      ,E_{1,1}} = 5 $
\end{enumerate}
\end{claim}

\begin{prop}
$X_{1}$ decomposes to the following $\GGt$ orbits:
\[O_{I      ,E_{1,1}},
O_{I      ,E_{1,2}},
O_{E_{1,2},E_{1,1}},
O_{E_{1,2},E_{1,2}},
O_{E_{1,2},E_{2,1}},
O_{E_{1,2},E_{2,2}} \]
\end{prop}

\begin{proof}
We can decompose $X_1$ to the $\GGt$-invariant subsets
\[X_{1,\text{inv}}=\left\{\left(x,y\right)\in X_{1} \,|\, x \text{ or } y \text{ are invertible}\right\}\] 
\[X_{1,\text{noninv}} = \left\{\left(x,y\right)\,|\, x,y \text{ are non-invertible}\right\}\]
Let $\left(x,y\right)\in X_{1,\text{inv}}$. It is equivalent under the action of $\ZZ[2]$ to $\left(y,x\right)$ so we may assume $x$ is invertible. Thus the point is equivalent to $\left(I,x^{-1}y\right)$.
$x^{-1}y$ is conjugate over $\overline{F}$ to a non-invertible Jordan matrix $J$, which is either 
\[J=\left[\begin{array}{cc} \lambda & 0 \\ 0 & 0 \end{array}\right]
  =\left[\begin{array}{cc} 1 & 0 \\ 0 & 0 \end{array}\right]=E_{1,1}\]
or
\[J=\left[\begin{array}{cc} 0 & 1 \\ 0 & 0 \end{array}\right]=E_{1,2}\] 
As both cases are defined over $F$, $x^{-1}y$ is conjugate over $F$ to one of $E_{1,1}$ or $E_{1,2}$.
In summary $X_{1,\text{inv}}$ decomposes as $O_{\left(I,E_{1,1}\right)}$ and $O_{\left(I,E_{1,2}\right)}$.

Let $\left(x,y\right)\in X_{1,\text{noninv}}$. We may assume $x$ is of Jordan form as before, i.e. $x=E_{1,2}$ or $x=E_{1,1}$. By replacing $x$ with $x\cdot S$ we may assume $x=E_{1,2}$. As $\Stab[\GG]{E_{1,2}} = B\times B$ we may multiply $y$ from left and from right by upper triangular matrices, i.e. we may multiply each column and each row of $y$ by a scalar, add the first row to the second and add the first column to the second. Since $y$ is not invertible, either its first row is zero or the second row is linearly dependent on the first row. In the second case we may add the first row multiplied by some scalar to the second so we may assume one of the rows of $y$ is zero. By the same argument, one of $y$'s columns is zero. We got that $\left(x,y\right)$ is equivalent to one of the following:
\[ \left( E_{1,2}, E_{1,1} \right),
  \left( E_{1,2}, E_{1,2} \right),
  \left( E_{1,2}, E_{2,1} \right),
  \left( E_{1,2}, E_{2,2} \right)\]
In each of those cases $\swap$ is represented by an element of $\GG$, so their orbits under $\GGt$ are the same as their orbits under $\GG$. Since 
\[ \dim O_{E_{1,2}, E_{1,1}} < 
  \dim O_{E_{1,2}, E_{1,2}} = 
  \dim O_{E_{1,2}, E_{2,1}} < 
  \dim O_{E_{1,2}, E_{2,2}}\]
we only have to prove that 
\[O_{E_{1,2},E_{1,2}}\neq O_{E_{1,2}, E_{2,1}}\]
or equivalently that 
\[ \left(B\times B\right)\act E_{1,2} \neq
  \left(B\times B\right)\act E_{2,1} \]
which is obvious as $E_{1,2}$ is an upper-triangular matrix and $E_{2,1}$ is not.
\end{proof}

\begin{prop} 
\label{GeomOfX1}
We have the following inclusions of closures of orbits:
\[\begin{aligned}
	&\overline{ O_{E_{1,2}, E_{1,2}} } \subseteq
	\overline{ O_{E_{1,2}, E_{1,1}} } \cap 	\overline{ O_{E_{1,2}, E_{2,2}} } \\
	&\overline{ O_{E_{1,2}, E_{1,1}} } \cup 	\overline{ O_{E_{1,2}, E_{2,2}} } \subseteq
  	\overline{ O_{E_{1,2}, E_{2,1}} } \cap 	\overline{ O_{I      , E_{1,2}} } \\
  	&\overline{ O_{E_{1,2}, E_{2,1}} } \cup 	\overline{ O_{I      , E_{1,2}} } \subseteq
  	\overline{ O_{I      , E_{1,1}} } 
\end{aligned}\]
and there are no other inclusions.

\ \ \ \ \ \ \ \ \ \ \ \ \ \ \ \ \ \ \ \ \ \ \ \ 
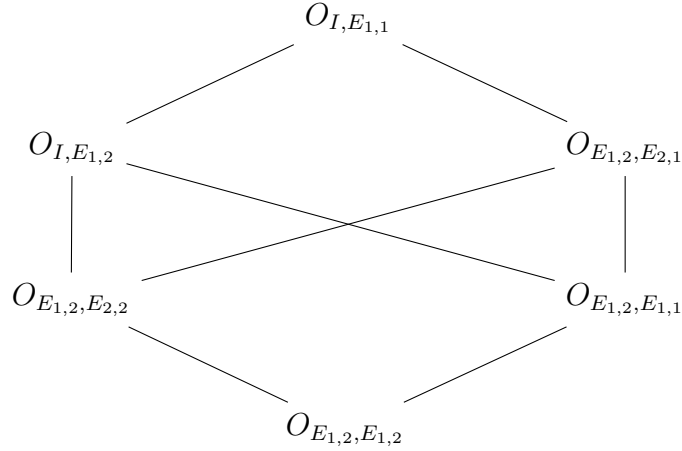
\begin{tikzpicture}[{baseline=(current bounding box.north)}]
\centering
\title{Figure 1}
\node(A)                           {$O_{I,E_{1,1}}$};
\node(B1)       [below right=1cm and 2cm of A] {$O_{E_{1,2},E_{2,1}}$};
\node(B2)       [below left=1cm and 2.2cm of A]  {$O_{I,E_{1,2}}$};
\node(C1)		[below right=3cm and 2cm of A] {$O_{E_{1,2},E_{1,1}}$};
\node(C2)		[below left=3cm and 2cm of A] {$O_{E_{1,2},E_{2,2}}$};
\node(D)		[below right=1cm and 1.75cm of C2] {$O_{E_{1,2},E_{1,2}}$};
\draw (A) -- (B1);
\draw (A) -- (B2);
\draw (B1) -- (C1);
\draw (B1) -- (C2);
\draw (B2) -- (C1);
\draw (B2) -- (C2);
\draw (C1) -- (D);
\draw (C2) -- (D);
\end{tikzpicture}

\end{prop}
\begin{proof}
\[\begin{aligned}
	\left( E_{1,2}, E_{1,2} \right) 
	& \in	\overline{\left\{
		E_{1,2}
	\right\} \times \left\{
		\alpha E_{1,1} + \beta E_{1,2} \,|\,
		\alpha \neq 0 
	\right\}} 												\\
	& = \overline{\left(B\times B\right)\act\left(E_{1,2}, E_{1,1} \right)} \\
	& \subseteq \overline{O_{E_{1,2},E_{1,1}}} 
\end{aligned}\]\\

\[\begin{aligned}
	\left( E_{1,2}, E_{1,2} \right) 
	& \in \overline{\left\{E_{1,2}\right\} \times 
		\left\{
			\alpha E_{2,2} + \beta E_{1,2} \,|\,
		\alpha \neq 0 
	\right\}}												\\
	& = \overline{\left(B\times B\right)\act\left(E_{1,2}\, E_{2,2} \right)}	\\
	& \subseteq \overline{O_{E_{1,2},E_{2,2}}}
\end{aligned}\]\\

\[\begin{aligned}
	\left(E_{1,2},E_{1,1}\right),	\left(E_{1,2},E_{2,2}\right) 
	& \in \{E_{1,2}\}\times \{\text{non invertbile matrices}\}	\\
	& = \overline{\left\{
		\left (
			E_{1,2}, 
			\left[\begin{array}{cc}
				a & b \\ 
				c & d
			\end{array}\right]	
		\right ) \,|\, 
		c\neq 0, ad=bc 
  	\right\}}													\\
  	& = \overline{\left(B\times B\right)\act \left(E_{1,2},E_{2,1}\right)}	\\
  	& \subseteq \overline{O_{E_{1,2},E_{2,1}}}
\end{aligned}\]\\

\[\begin{aligned}
	\left(E_{1,2},E_{1,1}\right),	\left(E_{1,2},E_{2,2}\right)
	& \in \left\{
		\left(E_{1,2},
		\left[\begin{array}{cc}
			* & * \\
			0 & *
		\end{array}\right]\right)
	\right\}					\\
	& = \overline{
		\left\{E_{1,2}\right\}\times
		B
	}											\\
	& = \overline{\left(B\times B\right)\act\left(E_{1,2},I\right)}	\\
	& \subseteq \overline{O_{I,E_{1,2}}}
\end{aligned}\]\\

\[\begin{aligned}
	\left(I,E_{1,2}\right)
	& \in \overline{ \left\{
		\left(I, \left[\begin{array}{cc}
			\delta & 1 \\
			\delta\epsilon & \epsilon
		\end{array}\right]\right)
		\,\bigg|\,
		\delta,\epsilon\neq 0
	\right\}}								\\
	& \subseteq \overline{\left\{
		\left( I, gE_{1,1}g^{-1}\right )
		\,|\,
		g\in G
	\right\}}								\\
	& \subseteq \overline{O_{I,E_{1,1}}} 
\end{aligned}\]\\

\[\begin{aligned}
	\left(E_{1,2},E_{2,1}\right)
	& \in \left\{ 
		\left(E_{1,2}, \left[\begin{array}{cc}
			* & * \\
			* & 0
		\end{array}\right]\right)
	\right\}													\\
	& = \overline{\left\{E_{1,2}\right\}\times \left(B\cdot S\right)}				\\
	& \subseteq \overline{\left(B\times B\right)\act \left(SE_{1,1}, S\right)}										\\
	& \subseteq \overline{O_{I,E_{1,1}}} 
\end{aligned}\]

A closure of an orbit can not contain a different orbit of an equal or higher dimension. Using \cref{stabilizers} we conclude that there can not be more inclusions.
\end{proof}
Let 
\[Z:= 
O_{E_{1,2},E_{1,2}} \cup 
O_{E_{1,2},E_{1,1}} \cup 
O_{E_{1,2},E_{2,2}} \cup 
O_{I      ,E_{1,2}} \]
it is a $\GGt$-set, and by \cref{GeomOfX1} it follows that $\overline{Z} = Z$, i.e. $Z$ is closed in $X_{1}$ and therefore in $\XX$. Denote by $P:=O_{E_{1,2},E_{2,1}}$, $W:=O_{I,E_{1,1}}$, $C:=X_{1} \setminus Z=W\cup P$, and $U:=C\cup O=\XX\setminus Z$. Then $U$ is open in $\XX$, $C$ is closed in $U$ and $P$ is closed in $C$.

\subsection{Reduction to $U$}
\begin{thm}
The restriction morphism $\Dist{\XX}^{\GGt,\chi}\to \Dist{U}^{\GGt,\chi}$ is an embedding. 
\end{thm}
\begin{proof}
By \cref{exactSeq} we have an exact sequence
\[0\to \Dist{Z} \to \Dist{\XX} \to \Dist{U} \to 0\]
Since $\left(\GGt,\chi\right)$-invariants is a left exact functor, we have the exact sequence
\[0\to \Dist{Z}^{\GGt,\chi} \to \Dist{\XX}^{\GGt,\chi} \to \Dist{U}^{\GGt,\chi}\]
Thus we wish to prove that $\Dist{Z}^{\GGt,\chi} = 0$. This is an easy use of the Bernstein-Gelfand-Kazhdan-Zelevinsky criterion (\cref{BGKZ}). Notice that each point in $Z$ has a stabilizer which is conjugate to a non-unimodular group and so is not unimodular. Let $\cal{F}$ be the sheaf on $Z$ associated to $\chi$. For any $z\in Z$:
\[ 
	\left(
		\cal{F}_{z} 						\otimes
		\Delta_{\GGt}	 						\otimes
		\Delta^{-1}_{\left(\GGt\right)_{z} }
	\right)^{\left(\GGt\right)_{z}} =
	\left( 
		\chi 									\otimes 
		\Delta^{-1}_{\left(\GGt\right)_{z} }
	\right)^{\left(\GGt\right)_{z}} = 0
\]
By Bernstein-Gelfand-Kazhdan-Zelevinsky:
\[
	\Dist{Z}^{\GGt,\chi} = 
	\Dist{Z, \cal{F}}^{\GGt} = 0
\]
\end{proof}
Therefore, to prove that $\Dist{\XX}^{\GGt,\chi}=0$ it is enough to prove that $\Dist{U}^{\GGt,\chi}=0$.

\subsection{Reduction to $C$}
In this subsection we will prove that \cref{C} follows from the following theorem:
\begin{thm}
\label{D}
$\Dist{C}^{\GGt} = 0$.
\end{thm}
The next section will be devoted to proving this theorem and thus ending the proof of \cref{mainThm}.

\begin{claim}
The space $\Dist{O}^{\GGt,\chi}$ is trivial.
\end{claim}

\begin{proof}
$\Dist{O}^{\GG}=\Dist{\GG}^{\GG}=\bb{C}\cdot\mu_{\GG}=\bb{C}\cdot \mu_{G}\boxtimes \mu_{G}$, where $\mu_{G}$ and $\mu_{\GG}$ are Haar measures of $G$ and $\GG$ respectively.

$\left(\mu_{G}\boxtimes \mu_{G}\right)^{\swap} = \mu_{G}\boxtimes \mu_{G}$, thus $\mu_{\GG}\in\Dist{O}^{\swap}$ and therefore $\Dist{O}^{\GG}\subseteq \Dist{O}^{\swap}$. By \cref{tildeIFF} $\Dist{O}^{\GGt}=0$.
\end{proof}

\begin{cor}
\label{distSuppOnC}
The inclusion maps
\[\begin{array}{c}
	\Dist{X_{1}}^{\GGt,\chi} \hookrightarrow \Dist{\XX}^{\GGt,\chi}		\\
	\Dist{C}^{\GGt,\chi} \hookrightarrow \Dist{U}^{\GGt,\chi}
\end{array}\]
are isomorphisms and in particular any distribution in $\Dist{\XX}^{\GGt,\chi}$ is supported on $X_{1}$ and any distribution in $\Dist{U}^{\GGt,\chi}$ is supported on $C$.
\end{cor}

\begin{proof}
By \cref{exactSeq} we have the exact sequences
\[0 \to \Dist{X_1} \to \Dist{\XX} \to \Dist{O} \to 0\]
\[0 \to \Dist{C} \to \Dist{U} \to \Dist{O} \to 0\]
Since $\left(\GGt,\chi\right)$-invariants is a left exact functor, we have the exact sequences
\[0 \to \Dist{X_1}^{\GGt,\chi} \to \Dist{\XX}^{\GGt,\chi} \to \Dist{O}^{\GGt,\chi}=0\]
\[0 \to \Dist{C}^{\GGt,\chi}  \to \Dist{U}^{\GGt,\chi} \to \Dist{O}^{\GGt,\chi}=0\]
\end{proof}

\begin{cor}
\cref{D} implies \cref{C}.
\end{cor}

\section{The Cross Method}
\label{crossMethod}
Let 
$\Pi := \left\{
	\left(
		\alpha E_{1,2} + E_{2,1} ,
		E_{1,2} + \beta E_{2,1}
	\right)
	\,|\,
	\alpha,\beta\in F
\right\}$
be a plane in $\XX$.

\subsection{The Geometry of $\Pi$ and its relation to $U$}
We will identify $\Pi$ with $\Aff[F]{2}$ under the identification 
\[
	\left ( \alpha, \beta \right )
	\mapsto \left (
		\left [ \begin{array}{cc}
			0 & \alpha \\
			1 & 0
		\end{array} \right ] ,
		\left [ \begin{array}{cc}
			0 & 1\\
			\beta & 0
		\end{array} \right ]
	\right )
\]
Under this identification, let $C_{\Pi}\subseteq\Aff[F]{2}$ be the subset corresponding to $\left\{xy=0\right\}\subseteq \Aff[F]{2}$, $P_{\Pi}$ be the singleton containing the origin and $W_{\Pi}=C_{\Pi}\setminus P_{\Pi}$.

$\Gm$ acts on $\Pi$ by 
\[\lambda\act\left(\alpha,\beta\right)=\left (\lambda\alpha, \lambda^{-1}\beta\right )\]
we expand the action to $\Gmt=\Gm\rtimes \ZZ[2]$, where $\ZZ[2]$ acts on $\Gm$ as inverse, by
\[\epsilon\act\left(\alpha,\beta\right)=\left(\beta,\alpha\right)\]
This action comes from the original $\GGt$ action on $\XX$, i.e. we have an embedding $\Gmt[F]\hookrightarrow \GGt$ such that $\Gmt[F]$ preserves $\Pi$ and acts as above. 
We embed $\Gmt[F]$ into $\GGt$ by 
\[\begin{aligned}
\left (\lambda, 1\right ) \mapsto \left(t_{\lambda},I, 1 \right)		\\
\left (\lambda, \epsilon\right ) \mapsto \left(S t_{\lambda},S, \epsilon \right)	
\end{aligned}\]	 
where $t_{\lambda}:= \left[ \begin{array}{cc} \lambda & 0 \\ 0 & 1 \end{array} \right]$, $S=\left [\begin{array}{cc} 0 & 1 \\ 1 & 0 \end{array}\right ]$. This is indeed an inclusion as $t_{\lambda\mu}=t_{\lambda}t_{\mu}$ and $St_{\lambda} = t_{\lambda^{-1}} S$.
Under this inclusion:
\[ \begin{aligned}
	\lambda \act \left (\alpha, \beta\right )
	& = 	\left ( t_{\lambda}, I \right ) 
	\act \left (
		\left [ \begin{array}{cc}
			0 & \alpha \\
			1 & 0
		\end{array} \right ] ,
		\left [ \begin{array}{cc}
			0 & 1 \\
			\beta & 0
		\end{array} \right ]
	\right )		\\
	& = \left (
		\left [ \begin{array}{cc}
			\lambda & 0 \\
			0 & 1
		\end{array} \right ] ,
		\left [ \begin{array}{cc}
			1 & 0 \\
			0 & 1
		\end{array} \right ]
	\right ) \act \left (
		\left [ \begin{array}{cc}
			0 & \alpha \\
			1 & 0
		\end{array} \right ],
		\left [ \begin{array}{cc}
			0 & 1 \\
			\beta & 0
		\end{array} \right ]
	\right )								\\
	& = \left (
		\left [ \begin{array}{cc}
			0 & \lambda \alpha \\
			1 & 0
		\end{array} \right ],
		\left [ \begin{array}{cc}
			0 & \lambda \\
			\beta & 0
		\end{array} \right ]
	\right )								\\
	& = \left (
		\left [ \begin{array}{cc}
			0 & \lambda \alpha \\
			1 & 0
		\end{array} \right ],
		\left [ \begin{array}{cc}
			0 & 1 \\
			\lambda^{-1} \beta & 0
		\end{array} \right ]
	\right )								
	= \left ( \lambda \alpha, \lambda^{-1} \beta \right )	\\
	\epsilon\act\left(\alpha,\beta\right)
	& = \left(S, S, \epsilon\right)
	\act \left(
		\left[\begin{array}{cc}
			0 & \alpha \\
			1 & 0
		\end{array}\right], 
		\left[\begin{array}{cc}
			0 & 1 \\
			\beta & 0
		\end{array}\right]
	\right)				\\
	& = \left(
		\left[\begin{array}{cc}
			0 & 1 \\
			1 & 0
		\end{array}\right ], 
		\left [\begin{array}{cc}
			0 & 1 \\
			1 & 0 
		\end{array}\right ]
		, \epsilon
	\right) \act \left(
		\left [ \begin{array}{cc}
			0 & \alpha \\
			1 & 0
		\end{array} \right ],
		\left [ \begin{array}{cc}
			0 & 1 \\
			\beta & 0 
		\end{array}\right ]
	\right)			\\
	& = \left(
		\left [ \begin{array}{cc}
			0 & \beta \\
			1 & 0
		\end{array}\right ],
		\left [ \begin{array}{cc}
			0 & 1 \\
			\alpha & 0
		\end{array}\right ]
	\right ) = \left(\beta, \alpha\right)		
\end{aligned} \]

Let $\Gmt[F]$ act on $\GGt$ by right multiplication. For $\Gmt[F]$-invariant subset $\Lambda\subseteq \Pi$, denote by $\widehat{\Lambda}:=\Lambda\times_{\Gmt[F]}\GGt$.
Let $a: \widehat{\Pi} \to \XX$ be given by the action of $\GGt$ on $\Pi\subseteq \XX$. It is obvious that $a$ does not depend on the choice of a representative and is $\GGt$-equivariant.

\begin{claim}
The image of $a$ lands in $U$. Moreover:
\[\begin{aligned}
	&a\left(\widehat{P_{\Pi}}\right)\subseteq P	\\
	&a\left(\widehat{W_{\Pi}}\right)\subseteq W	\\
	&a\left(\widehat{\Pi\setminus C_{\Pi}}\right)\subseteq O
\end{aligned}\]
\end{claim}

\begin{proof}
It is enough to prove the second part of the claim.
The action of $\ZZ[2]$ swaps elements, and as $P$, $W$ and $O$ are symmetric it is enough to prove that 
\[\begin{aligned}
	& a\left(
		P_{\Pi}\times_{\Gm[F]}\left(\GG\right)
	\right)	\subseteq P							\\
	& a\left(
		W_{\Pi}\times_{\Gm[F]}\left(\GG\right)
	\right)	\subseteq W							\\
	& a\left(
		\left(\Pi\setminus C_{\Pi}\right)
		\times_{\Gm[F]}\left(\GG\right)
	\right)	\subseteq O	
\end{aligned}\]						
Let $\left(\alpha,\beta\right)\in\Pi$ and $\left(g,h\right)\in\GG$.
\[
	a\left (
		\left (
			\alpha, \beta
		\right ), \left (
			g, h
		\right )
	\right ) 
	 = \left (
	  	g \left [
	  		\begin{array}{cc}
	  			0 & \alpha \\
	  			1 & 0
	  		\end{array}
	  	\right ] h^{-1}, 
	  	g \left [
	  		\begin{array}{cc}
	  			0 & 1 \\
	  			\beta & 0	  		
	  		\end{array}
	  	\right ] h^{-1}
	\right )
\]
The claim follows by the definition of $P,W,O$.
\end{proof}

As an easy consequence we get that $\Pi\subseteq U$.

\subsection{Restrictions from $P$ and $W$}
\label{restPW}
As the $\GGt$-equivariant, smooth maps $a\big|_{\widehat{P_{\Pi}}}$ and $a\big|_{\widehat{W_{\Pi}}}$ land in the $\GGt$-orbits $P$ and $W$ respectively, they are submersions.
Let $E_{P}$ and $E_{W}$ be $\GGt$-equivariant vector bundles over $P$ and $W$ respectively.
Denote the restrictions
$a_{P}:=a\big|_{\widehat{P_{\Pi}}}$,
$a_{W}:=a\big|_{\widehat{W_{\Pi}}}$.
Using Harish-Chandra's submersion principle, \cref{HCS}, we get surjective continuous linear maps
\[\begin{aligned}
& \left(a_{P}\right)_{*}: 
	\Sch{
		\widehat{P_{\Pi}},	
		a_{P}^{*}E_{P}\otimes \Dens{\widehat{P_{\Pi}}}
	} \twoheadrightarrow \Sch{
		P, 
		E_{P} \otimes \Dens{P}
	}											\\
& \left(a_{W}\right)_{*}:
	\Sch{
		\widehat{W_{\Pi}},	
		a_{W}^{*}E_{W}\otimes \Dens{\widehat{W_{\Pi}}}
	} \twoheadrightarrow \Sch{
		W, 
		E_{W} \otimes \Dens{W}
	}   											
\end{aligned}\]

that satisfy:
\[\begin{aligned}
	& \int_{P} \left\langle 
		f, 
		\left(a_{P}\right)_{*}\mu
	\right\rangle
	= \int_{\widehat{P_{\Pi}} }	\left\langle
		a_{P}^{*}f, 
		\mu 
	\right\rangle 
	\ \forall f\in \Sch{P, E_{P}}, 
	\mu\in \Sch{
		\widehat{P_{\Pi}},	
		a_{P}^{*}E_{P}\otimes \Dens{\widehat{P_{\Pi}}}
	}													\\
	& \int_{W} \left\langle 
		f, 
		\left(a_W\right)_{*}\mu
	\right\rangle
	= \int_{\widehat{W_{\Pi}} }	\left\langle
		a_{W}^{*}f, 
		\mu 
	\right\rangle 
	\ \forall f\in \Sch{W, E_{W}}, 
	\mu\in \Sch{
		\widehat{W_{\Pi}},	
		a_{W}^{*}E_{W}\otimes \Dens{\widehat{W_{\Pi}}}
	}
\end{aligned}\]
	
Taking the dual maps we get the injections:
\[\begin{aligned}
& a_{P}^{*}:
	\genfun{
		P, E_{P}
	} \hookrightarrow \genfun{
		\widehat{P_{\Pi}}, a_{P}^{*}E_{P}
	}									\\
& a_{W}^{*}:
	\genfun{
		W, E_{W}
	} \hookrightarrow \genfun{
		\widehat{W_{\Pi}}, a_{W}^{*}E_{W}
	}
\end{aligned}\]
Which, by its definition is a (and whence the only) continuous extension of the push forward of smooth functions.

These maps are $\GGt$-equivariant and $\left(\GGt,\chi\right)$-invariants is a left exact functor, we have injective maps
\[\begin{aligned}
& a_{P}^{*}:
	\genfun{
		P, E_{P}
	}^{\GGt,\chi} \hookrightarrow \genfun{
		\widehat{P_{\Pi}}, a_{P}^{*}E_{P}
	}^{\GGt,\chi} 							\\
& a_{W}^{*}:
	\genfun{
		W, E_{W}
	}^{\GGt,\chi} \hookrightarrow \genfun{
		\widehat{W_{\Pi}}, a_{W}^{*}E_{W}
	}^{\GGt,\chi}
\end{aligned}\]

We look at the projection maps 
\[\widehat{P_{\Pi}}\to \sfrac{\GGt}{\Gmt[F]}, \widehat{C_{\Pi}}\to \sfrac{\GGt}{\Gmt[F]}\]
Which have fibers $P_{\Pi}$ and $C_{\Pi}$ respectively. Using Frobenius Descent (\cref{frobDesc}) we have canonical isomorphisms
\[\begin{aligned}
	&\Fr_{P}: \genfun{P_{\Pi}, \left(a_{P}^{*}E_{P}\right)\big|_{P_{\Pi}}}^{\Gmt[F],\chi}
	\xrightarrow{\sim} \genfun{\widehat{P_{\Pi}}, a_{P}^{*}E_{P}}^{\GGt,\chi}		\\
	&\Fr_{W}: \genfun{W_{\Pi}, \left(a_{W}^{*}E_{W}\right)\big|_{W_{\Pi}}}^{\Gmt[F],\chi}
	\xrightarrow{\sim} \genfun{\widehat{W_{\Pi}}, a_{W}^{*}E_{W}}^{\GGt,\chi}
\end{aligned}\]

Composing $a_{W}^{*}$ and $a_{P}^{*}$ with the isomorphisms given by Frobenius decent we get injective "restriction" maps
\[\begin{aligned}
	&\genfun{
		P, a_{P}^{*}E_{P}
	}^{\GGt,\chi} \hookrightarrow \genfun{
		P_{\Pi}, \left(a_{P}^{*}E_{P}\right)\big|_{P_{\Pi}}
	}^{\Gmt[F],\chi} 							\\
	&\genfun{
		W, a_{W}^{*}E_{W}
	}^{\GGt,\chi} \hookrightarrow \genfun{
		W_{\Pi}, \left(a_{W}^{*}E_{W}\right)\big|_{W_{\Pi}}
	}^{\Gmt[F],\chi}
\end{aligned}\]

Let $\iota_{P}: P_{\Pi}\hookrightarrow \widehat{P_{\Pi}}$ and $\iota_{W}: W_{\Pi}\hookrightarrow \widehat{W_{\Pi}}$ be the natural inclusions.

\begin{claim}
$\left(a_{P}^{*}E_{P}\right)\big|_{P_{\Pi}}=E_{P}\big|_{P_{\Pi}}$,
$\left(a_{W}^{*}E_{W}\right)\big|_{W_{\Pi}}=E_{W}\big|_{W_{\Pi}}$,
\end{claim}

\begin{proof}
Let $i_{1}:P_{\Pi}\hookrightarrow P$. Then $i_{P}= a_{P}\circ\iota_{P}$.
Thus
\[
E_{P}\big|_{P_{\Pi}}
= i_{P}^{*} E_{P}
= \iota_{P}^{*}\left(
	a_{P}^{*}\left(
		E_{P}
	\right)
\right)
= \left(a_{P}^{*}E_{P}\right)\big|_{P_{\Pi}}
\]
An analogous argument works for $W$.
\end{proof}

\begin{cor}
\label{corRestFromPW}
We have injective "restriction" maps
\[\begin{aligned}
	& \genfun{
		P, a_{P}^{*}E_{P}
	}^{\GGt,\chi} \hookrightarrow \genfun{
		P_{\Pi}, E_{P}\big|_{P_{\Pi}}
	}^{\Gmt[F],\chi} 							\\
	& \genfun{
		W, a_{W}^{*}E_{W}
	}^{\GGt,\chi} \hookrightarrow \genfun{
		W_{\Pi}, E_{W}\big|_{W_{\Pi}}
	}^{\Gmt[F],\chi}
\end{aligned}\]
\end{cor}

\begin{cor}
\label{bundlesOnPTriv}
Any $\GGt$-equivariant vector bundle $E_{P}$ over $P$ such that $E_{P}\big|_{P_{\Pi}}$ is $\ZZ[2]$-invariant satisfies:
\[\Dist{P,E_{P}}^{\GGt,\chi}=0\]
\end{cor}

\begin{proof}
We look at
\[\Dist{
	P,E_{P}
}^{\GGt,\chi} 
= \genfun{
	P,
	E_{P}^{*}
	\otimes \Dens{P}
}^{\GGt,\chi}\]
which embeds into 
\[\genfun{
	P_{\Pi}, 
	E_{P}\big|_{
		P_{\Pi}
	} \otimes \Dens{P}\big|_{
		P_{\Pi}
	}
}^{\Gmt[F],\chi}\]
As $P_{\Pi}$ is a point, $\Dens{P}\big|_{	P_{\Pi}	}$ is $\ZZ[2]$-invariant. I.e. $\genfun{
	P_{\Pi}, 
	E_{P}\big|_{
		P_{\Pi}
	} \otimes \Dens{P}\big|_{
		P_{\Pi}
	}
}$ is a $\Gmt[F]$-representation on which $\ZZ[2]$ acts trivially. As $\chi$ is a non-trivial character on $\ZZ[2]$:
\[\genfun{P_{\Pi}, E_{P}\big|_{P_{\Pi}} \otimes\Dens{P}\big|_{	P_{\Pi}	}}^{\Gmt[F],\chi}=0\]
\end{proof}

Let $\eta$ be the composition
\[
\begin{aligned}
	\Dist{C}^{\GGt,\chi}
	& \to \Dist{W}^{\GGt,\chi} = \genfun{W,\Dens{W}}^{\GGt,\chi}		\\
	& \hookrightarrow \genfun{W_{\Pi}, \Dens{W}\big|_{W_{\Pi}}}^{\Gmt[F], \chi}
\end{aligned}\]
where the first map is just the regular restriction of distributions and the second map is the restriction map built in \cref{corRestFromPW}.

\begin{claim}
$\eta$ is injective.
\end{claim}

\begin{proof}
It is enough to prove that $\Dist{C}^{\GGt,\chi}\to\Dist{W}^{\GGt,\chi}$ is injective.
By \cref{exactSeq} we have the exact sequence
\[
	0\to\Dist{P}\to\Dist{C}\to\Dist{W}\to 0
\]
As $\left(\GGt,\chi\right)$ is left exact we have the exact sequence
\[
	0
	\to \Dist{P}^{\GGt,\chi}
	\to \Dist{C}^{\GGt,\chi}
	\to \Dist{W}^{\GGt,\chi}
\]
By \cref{bundlesOnPTriv}, $\Dist{P}^{\GGt,\chi}=0$.
\end{proof}

\subsection{Restriction from $U$}
\begin{thm} 
\label{WFcapPi}
Let $\xi\in\Dist{U}^{\GGt,\chi}$. Then $\WF{\xi}\cap \normal[U]{\Pi} = \Pi\times \left\{0\right\}$.
\end{thm}

\begin{proof}
To end abbrevity, now denote elements of $G$ by $\left[g\right]$ and elements of $X$ by $\left[x\right]$.

Let $\widetilde{U}$ be the inverse image of $U$ in 
$\left( \gl[F]{2} \setminus \left\{0\right\} \right) \times 
 \left( \gl[F]{2} \setminus \left\{0\right\} \right)$ under the projection map 
$\pr: \left( \gl[F]{2} \setminus \left\{0\right\} \right) 
\times \left( \gl[F]{2} \setminus \left\{0\right\} \right) 
\twoheadrightarrow \XX$.

Let 
$\Omega= \left\{
	\left(
		\alpha E_{1,2} + E_{2,1} ,
		E_{1,2} + \beta E_{2,1}
	\right)
	\,|\,
	\alpha\beta \neq 1
\right\}$
be an open dense set in $\Pi$.
Let $\left( \left[x\right], \left[y\right] \right)\in \Omega$ and 
$x = \left(\begin{array}{cc} 0 & \alpha \\ 1 & 0 \end{array}\right)$,
$y = \left(\begin{array}{cc} 0 & 1 \\ \beta  & 0 \end{array}\right)$
lifts in $\gl{2}$. Notice that $\left(x,y\right)\in \tilde{U}$. We have a commutative square
\[ \xymatrix@C=4pc{
	\GL[F]{2}\times \GL[F]{2} 
		\ar[r]\sp(0.7){\widetilde{a}_{x,y}}
		\ar@{->>}[d]^{\pr} 
		& 
	\widetilde{U}
		\ar@{->>}[d]^{\pr}
		\\
	\GG
		\ar[r]^{a_{\left[x\right],\left[y\right]}} 
		& 
	U
}\]
Where $a_{\left[x\right],\left[y\right]}\left(\left[g\right],\left[h\right]\right) = \left(\left[gxh^{-1}\right],\left[gyh^{-1}\right]\right)$ is the action map of $\left(\left[x\right],\left[y\right]\right)$ and $\widetilde{a}_{x,y}\left(g,h\right) = \left(gxh^{-1},gyh^{-1}\right)$.
Differentiating this diagram we get:
\[\xymatrix@C=4pc{
	\gl[F]{2} \times \gl[F]{2}
		\ar[r]\sp(0.55){d\widetilde{a}_{x,y}}
		\ar@{->>}[d]^{d\pr} 
		& 
	T_{\left(x,y\right)} \tilde{U}
		\ar@{->>}[d]^{d\pr}
		\\
	\mathfrak{g} \times \mathfrak{g}
		\ar[r]^{da_{\left[x\right],\left[y\right]}} 
		& 
	T_{\left(\left[x\right],\left[y\right]\right)}U
}\]
As $\widetilde{U}$ is open in $\gl[F]{2}\times \gl[F]{2}$ we can identify $T_{\left(x,y\right)}\widetilde{U}=\gl[F]{2}\times\gl[F]{2}$. Under this identification $d\widetilde{a}_{x,y}$ is given by 
$d\widetilde{a}_{x,y}\left(A,B\right) = 
\left( Ax-xB, Ay-yB \right)$.

Let $\varphi\in\WF[ {\left(\left[x\right],\left[y\right]\right)} ]{\xi} \cap \normal[U]{\Pi,\left(\left[x\right],\left[y\right]\right)} $.
Look at $\left(d\pr\right)^{*}\varphi\in \left(\gl[F]{2}\times \gl[F]{2}\right)^{*}$. Let $\widetilde{\Pi}$ be the pullback of $\Pi$ along $\pr$:
\[
	\widetilde{\Pi}
	:= \pr^{-1}\left ( \Pi \right )
	= \left \{
		\left ( \begin{array}{cc}
			0 & * \\
			* & 0
		\end{array} \right ) ,
		\left ( \begin{array}{cc}
			0 & * \\
			* & 0
		\end{array} \right )
	\right \}
	\subseteq \left ( 
		\gl[F]{2} \setminus \left \{ 0 \right \}
	\right ) \times \left (
		\gl[F]{2} \setminus \left \{ 0 \right \}
	\right )
\]
Then for $v\in T_{\left(x,y\right)} \widetilde{\Pi}$:
\[
	\left (d \pr \right )^{*} \varphi \left ( v \right )
	= \varphi \left ( d\pr \left ( v \right ) \right )
	\in \varphi \left ( T_{\left[x\right],\left[y\right]}\Pi \right )
	= 0
\]
Thus $\left ( d\pr \right )^{*}\left (\varphi \right ) \in \left ( \widetilde{\frak{t}}\times\widetilde{\frak{t}} \right )^{*}$, where 
$\widetilde{\frak{t}} = \left \{ \left ( \begin{array}{cc}
	* & 0 \\
	0 & *
\end{array} \right ) \right \}
\subseteq \gl[F]{2}$.

Now let $\left ( t,s \right ) \in \widetilde{\frak{t}}\times\widetilde{\frak{t}}$.
Write 
$
	t = \left ( \begin{array}{cc}
		t_{1} & 0 \\
		0 & t_{2}
	\end{array} \right )
$,
$
	s = \left ( \begin{array}{cc}
		s_{1} & 0 \\
		0 & s_{2}
	\end{array} \right )
$.
Choose 
\[
	A = \frac{1}{\alpha\beta-1} \left ( \begin{array}{cc}
		0 & \alpha s_{1} - t_{1} \\
		\beta t_{2} - s_{2} & 0
	\end{array} \right ),
	B = \frac{1}{\alpha\beta-1} \left ( \begin{array}{cc}
		0 & t_{2} - \alpha s_{2} \\
		s_{1} - \beta t_{1} & 0
	\end{array} \right )
\]
Then $d\widetilde{a} \left ( A,B \right ) = \left ( t,s \right )$.
\[
	\left ( d\pr \right )^{*} \varphi \left ( t,s \right )
	= \varphi \left ( 
		d\pr \left ( 
			d\widetilde{a}_{x,y} \left ( A,B \right )
		\right )
	\right )
	= \varphi \left ( 
		da_{\left[x\right],\left[y\right]} \left (
			d\pr \left ( A,B \right ) 	
		\right )
	\right )
	= 0
\]
As $\varphi \in \WF[{\left([x],[y]\right)}]{\xi} \subseteq \left ( da_{[x],[y]} \left ( \frak{g}\times\frak{g} \right ) \right )^\perp$ by \cref{action-WF}.

Therefore $\left ( d\pr \right )^{*}\varphi \perp \gl[F]{2}\times \gl[F]{2}$, i.e. $\left ( d\pr \right )^{*}\varphi = 0$. As $d\pr$ is surjective, $\left ( d\pr \right )^{*}$ is injective, therefore $\varphi = 0$.
We proved that for any $\left ( \left[x\right],\left[y\right] \right )\in \Omega$, 
$
	\left ( 
		\WF{\Pi} \cap \normal[U]{\Pi} 
	\right )_{\left(\left[x\right],\left[y\right]\right)}
	= \WF[{\left(\left[x\right],\left[y\right]\right)}]{\xi} \cap 
		\normal[U]{\Pi, \left(\left[x\right],\left[y\right]\right)}
	= 0
$. I.e. 
\[ 
	\WF{\xi}\cap \normal[U]{\Pi} 
	\subseteq \Pi\times\{0\} \cup T^{*} U\big|_{\Pi\setminus\Omega} 
\]
By \cref{distSuppOnC}, $\supp[\xi]\subseteq C$. As $C\cap \left(\Pi\setminus\Omega\right) = \emptyset$:
\[\WF{\xi}\cap\normal[U]{\Pi}\subseteq \Pi\times \{0\}\]
\end{proof}

\begin{cor}
$\Dist{C}^{\GGt,\chi}=\Dist{U}^{\GGt,\chi}\subseteq\Dist{U}^{\GG}=\genfun{U,\Dens{U}}^{\GG}\subseteq \genfun[\left({\normal[U]{\Pi}}\right)^{c}]{U,\Dens{U}}$.
\end{cor}

\subsection{End of Proof}
\begin{thm}
\label{important}
There exist $\Gmt[F]$-equivariant maps of the $\Gmt[F]$-spaces:
\[
\Dist{C} \cap \genfun[{\left(\normal[U]{\Pi}\right)^{c}}]{U,\Dens{U}}
\to \Dist{C_{\Pi}}
\to \genfun{W_{\Pi},\Dens{W}\big|_{W_{\Pi}}}
\]
and the composition 
\[\Dist{C}^{\GGt,\chi}\subseteq 
 \Dist{C}\cap\genfun[{\left(\normal[U]{\Pi}\right)}]{U,\Dens{U}} \to
 \Dist{C_{\Pi}} \to
 \genfun{W_{\Pi},\Dens{W}\big|_{W_{\Pi}}}\]
agrees with $\eta$.
\end{thm}

To prove \cref{important} we will need the following lemma:
\begin{lem}
\label{vectBunTriv}
The $\Gmt[F]$-equivariant line bundles
	$\Dens{U}$ over $U$,
	$\Dens{\Pi}$ over $\Pi$,
	$\Dens{W}$ over $W$,
	$\Dens{W_{\Pi}}$ over $W_{\Pi}$	and
	$\Dens{\widehat{W_{\Pi}}}$ over $\widehat{W_{\Pi}}$
are trivial.
\end{lem}

\begin{proof}
$\Dens{U}=\Dens{\XX}\big|_{U}$ as $U$ is open in $\XX$, so it will be enough to prove that $\Dens{\XX}=\Dens{X}\otimes\Dens{X}$ is $\Gmt[f]$-invariant and whence it is enough to show that $\Dens{X}$ is $\Gm[F]$-invariant.
$X=\Pgl[F]{2}\cong F\mathbb{P}^{3}$. We let the coordinates on $X$ be $\left [
	\begin{array}{cc}
		x & y \\
		z & w
	\end{array}
\right ]$
then on the open set $\left \{xyzw\neq 0\right \}$ we have a top differential form of the form
\[
  \frac{dx\wedge dy\wedge dz}{xyz}
- \frac{dx\wedge dy\wedge dw}{xyw}
+ \frac{dx\wedge dz\wedge dw}{xzw}
- \frac{dy\wedge dz\wedge dw}{yzw}
\]
This differential form can be extended to all of $X$. As $\Gm[F]$ acts on $X$ by $\lambda\act\left[x:y:z:w\right]=\left[\lambda x:\lambda y:z:w\right]$, this top differential form is $\Gm[F]$-invariant. Thus $\det(X)$ is $\Gm[F]$-invariant and therefore so is $\Dens{X}$.
\\\\
$\det(\Pi)\cong\det(\Aff[F]{2})=F\cdot dx\wedge dy$. For $\lambda\in \Gm[F]$:
\[\lambda\act\left (dx\wedge dy\right )
=\left (\lambda dx\right )\wedge \left (\lambda^{-1} dy\right )
=dx \wedge dy\]
And $\epsilon\act\left(dx\wedge dy\right)=dy\wedge dx = -dx\wedge dy$, and therefore $\epsilon\act\left|dx\wedge dy\right|=\left|dx\wedge dy\right|$.
We conclude that $\Dens{\Pi}$ is a $\Gmt[F]$-invariant vector bundle.
\\\\
$W=O_{I,E_{1,1}}$ is a $\GGt$-orbit with stabilizer $P_{I,E_{1,1}}=\Delta T$ (\cref{stabilizers}) which is unimodular. Thus $\Dens{W}$ is trivial as a $\GGt$-equivariant vector bundle, and in particular as a $\Gmt[F]$-equivariant bundle.
\\\\
$W_{\Pi}$ is a $\Gmt[F]$-torsor, and thus $\Dens{W_{\Pi}}$ is a trivial $\Gmt[F]$-equivariant bundle.
\\\\
As $W_{\Pi}$ is a $\Gmt[F]$-torsor, $\widehat{W_{\Pi}}=W_{\Pi}\times_{\Gmt[F]}\GGt[F]$ is a $\GGt$-orbit with stabilizer $\Stab[\GGt]{\widehat{W_{\Pi}}}=\Stab[{\Gmt[F]}]{W_{\Pi}}=\Gmt[F]$ which is unimodular. Therefore $\Dens{\widehat{W_{\Pi}}}$ is a trivial $\GGt$-equivariant bundle and in particular a trivial $\Gmt[F]$-equivariant bundle.

\end{proof}

\begin{proof}[Proof of \cref{important}]
Using \cref{rest-crit} we get a restriction map 
\[\genfun[{\left(\normal[U]{\Pi}\right)^{c}}]{U,\Dens{U}} \to \genfun{\Pi, \Dens{U}\big|_{\Pi}}=\Dist{\Pi,\Dens{\Pi}\otimes\Dens{U}^{*}\big|_{\Pi}}\]
which extends the restriction of smooth sections \[\Sm{U,\Dens{U}}\to\Sm{\Pi,\Dens{U}\big|_{\Pi}}\]
By \cref{vectBunTriv} we can write this map as: 
\[\genfun[{\left(\normal[U]{\Pi}\right)^{c}}]{U,\Dens{U}} \to \Dist{\Pi}\]

As the restriction preserves support, the image of $\Dist{C}\cap \genfun[{\left(\normal[U]{\Pi}\right)^{c}}]{U,\Dens{U}}$ under this map lands in $\Dist{C_{\Pi}}$.
We now have the commutative diagram of $\Gmt[F]$-spaces:
\[\xymatrix@C=0.7pc{
	\Dist{C}\cap \genfun[{\left(\normal[U]{\Pi}\right)^{c}}]{U,\Dens{U}}
		\ar[r]
		\ar[rrd]
		&
	\genfun[{\left(\normal[U]{\Pi}\right)^{c}}]{U,\Dens{U}}
		\ar[r]
		&
	\Dist{\Pi}
		\\
		&
		&
	\Dist{C_{\Pi}}
		\ar[u]
}\]

We look at the restriction map of distributions
\[\Dist{C}\cap \genfun[{\left(\normal[U]{\Pi}\right)^{c}}]{U,\Dens{U}}\to\Dist{W}\]
and at the map $a_{w}^{*}$ from \ref{restPW}:
\[
\Dist{W}
 = \genfun{W,\Dens{W}}
 \to \genfun{\widehat{W_{\Pi}}, a_{W}^{*}\Dens{W}}
\]
which is an extension of the pullback by $a_{W}$ of smooth sections
\[a_{W}^{*}: \Sm{{W,\Dens{W}}}\to\Sm{{\widehat{W_{\Pi}},a_{W}^{*}\Dens{W}}}\]
By its definition, $\eta$ is the induced map by applying $\left(\GGt,\chi\right)$-invariants to the composition
\[
\Dist{C}
\to \Dist{W}
\to \genfun{\widehat{W_{\Pi}},a_{W}^{*}\Dens{W}}
\]
We have built the diagram
\[\xymatrix@C=0.7pc{
	\Dist{C}\cap \genfun[{\left(\normal[U]{\Pi}\right)^{c}}]{U,\Dens{U}}
		\ar[r]
		\ar[rrd]
		\ar[dd]
		&
	\genfun[{\left(\normal[U]{\Pi}\right)^{c}}]{U,\Dens{U}}
		\ar[r]
		&
	\Dist{\Pi}
		\\
		&
		&
	\Dist{C_{\Pi}}
		\ar[u]
		\\
	\Dist{W}
		\ar[r]
		&
	\genfun{\widehat{W_{\Pi}},a_{W}^{*}\Dens{W}}
		&
}\]

We have a restriction map 
\[\Dist{C_{\Pi}}\to \Dist{W_{\Pi}}\]
which extends the restriction of measures 
\[\{\mu\in\Sm{U,\Dens{U}} \text{ supported on } C_{\Pi}\}\to \Sm{W_{\Pi},\Dens{W_{\Pi}}}\]

Using \cref{vectBunTriv} we can write this map as:
\[\Dist{C_{\Pi}}\to\genfun{W_{\Pi}, \Dens{W}\big|_{W_{\Pi}}}\]

$W_{\Pi}$ embeds into $\widehat{W_{\Pi}}$ as a closed subspace.
Thus we have a map
\[\begin{aligned}
	e: \genfun{W_{\Pi}, \Dens{W}\big|_{W_{\Pi}}}
	& = \Dist{W_{\Pi}, \Dens{W_{\Pi}}\otimes\Dens{W}}			
	 \cong \Dist{W_{\Pi}}										\\
	& \to   \Dist{\widehat{W_{\Pi}}}							
	 =  \genfun{\widehat{W_{\Pi}}, \Dens{\widehat{W_{\Pi}}}}	\\
	& \cong  \genfun{\widehat{W_{\Pi}}, a_{W}^{*}\Dens{W}}
\end{aligned}\]
Which is an extension of the continuation of smooth measures.
We got the following diagram:
\[\xymatrix@C=0.7pc{
	\Dist{C} \cap \genfun[{\left(\normal[U]{\Pi}\right)^{c}}]{U,\Dens{U}}
		\ar[r]
		\ar[rrd]
		\ar[dd]
		&
	\genfun[{\left(\normal[U]{\Pi}\right)^{c}}]{U,\Dens{U}}
		\ar[r]
		&
	\Dist{\Pi}
		\\
		&
		&
	\Dist{C_{\Pi}}
		\ar[u]
		\ar[d]
		\\
	\Dist{W}
		\ar[r]
		&
	\genfun{\widehat{W_{\Pi}}, a_{W}^{*}\Dens{W}}
		&
	\genfun{W_{\Pi},\Dens{W}\big|_{W_{\Pi}}}
		\ar[l]_{e}
}\]
It is commutative as the diagram of subspaces
\[\xymatrix@C=0.2pc{
	{\left\{\begin{aligned}
	&\mu\in\Sm{U,\Dens{U}} \\
	&\text{ supported on } C_{\Pi}
	\end{aligned}\right\}}
		\ar[r]
		\ar[rrd]
		\ar[dd]
		&
	\Sm{U,\Dens{U}}
		\ar[r]
		&
	{\begin{aligned}
	&\Sm{\Pi, \Dens{\Pi}\otimes\Dens{U}^{*}\big|_{\Pi}}\\
	&\cong \Sm{\Pi,\Dens{U}\big|_{\Pi}}
	\end{aligned}}
		\\
		&
		&
	{\left\{\begin{aligned}
	&\mu\in\Sm{U,\Dens{U}} \\
	&\text{ supported on } C_{\Pi}
	\end{aligned}\right\}}
		\ar[u]
		\ar[d]
		\\
	{\begin{aligned}
	&\Sm{W,\Dens{U}\big|_{W}} \\
	&\cong \Sm{W,\Dens{W}}
	\end{aligned}}
		\ar[r]
		&
	\Sm{\widehat{W_{\Pi}}, a_{W}^{*}\Dens{W}}
		&
	\Sm{W_{\Pi},\Dens{W}\big|_{W_{\Pi}}}
		\ar[l]
}\]
is commutative. These subspaces are dense in the original spaces, and whence the original diagram commutes. Let $\eta'$ be the composition 
\[\Dist{C}^{\GGt,\chi}\subseteq 
 \Dist{C}\cap\genfun[{\left(\normal[U]{\Pi}\right)}]{U,\Dens{U}} \to
 \Dist{C_{\Pi}} \to
 \genfun{W_{\Pi},\Dens{W}\big|_{W_{\Pi}}}\]
We know that the map $\Dist{C}^{\GGt,\chi}\to\genfun{\widehat{W_{\Pi}},a_{W}^{*}\Dens{W}}$ is the composition of $e'\circ\eta$, where $e'$ is the inclusion 
\[\begin{aligned}
	\genfun{W_{\Pi},\Dens{W}\big|_{W_{\Pi}}}^{\Gmt[F],\chi}
	& = \genfun{\widehat{W_{\Pi}},a_{W}^{*}\Dens{W}}^{\GGt,\chi}		\\
	& \hookrightarrow  \genfun{\widehat{W_{\Pi}},a_{W}^{*}\Dens{W}}
\end{aligned}\]
Then, by the commutativity of the diagram, the above composition also equals to $e\circ \eta'$. $e'$ decomposes as $e\circ e''$, where $e''$ is the embedding
\[
	\genfun{W_{\Pi},\Dens{W}\big|_{W_{\Pi}}}^{\Gmt[F],\chi}
	\hookrightarrow \genfun{W_{\Pi},\Dens{W}\big|_{W_{\Pi}}}
\]
Thus
\[
e\circ e''\circ \eta = e'\circ \eta = e\circ \eta'
\]
As $e$ is injective, $e''\circ \eta=\eta'$
\end{proof}

\begin{proof}[Proof of \cref{D}]
Using the notations of the proof, we have a map 
\[\begin{aligned}
	\eta' = e''\circ \eta:
	\Dist{C}^{\GGt,\chi}
	& \subseteq \Dist{C} \cap \genfun[{\left(\normal[U]{\Pi}\right)^{c}}]{U}		\\
	& \to \Dist{C_{\Pi}}															\\
	& \to \genfun{W_{\Pi},\Dens{W}\big|_{W_{\Pi}}}
\end{aligned}\] 
As all maps, apart from the inclusion, are defined over $\Gmt[F]$ and the inclusion still holds for the $\left(\Gmt[F],\chi\right)$-invariants, we get a map
\[\begin{aligned}
	\left(e''\right)^{\Gmt[F],\chi}\circ \eta: 
	\Dist{C}^{\GGt,\chi}
	& \subseteq \left(\Dist{C} \cap \genfun[{\left(\normal[U]{\Pi}\right)^{c}}]{U}\right)^{\Gmt[F],\chi}		\\
	& \to \Dist{C_{\Pi}}^{\Gmt[F],\chi}				\\
	& \to \genfun{W_{\Pi},\Dens{W}\big|_{W_{\Pi}}}^{\Gmt[F],\chi}
\end{aligned}\] 
By its definition $e''$ is the identity map. And so we get a decomposition of $\eta$ that goes through $\Dist{C_{\Pi}}^{\Gmt,\chi}$. Under our identification $\Pi\cong\Aff[F]{2}$, $C_{\Pi}=\{xy=0\}$. By \cref{realCross}, $\Dist{\{xy=0\}}^{\Gmt,\chi}=0$, and whence $\eta$ decomposes through a null space, i.e. $\eta = 0$. As $\eta$ is injective, $\Dist{C}^{\GGt,\chi}=0$.
\end{proof}

\newpage

\bibliography{mybib}{}

\begin{thebibliography}{vdBCD96}

\bibitem[AD15]{Aiz2}
Avraham Aizenbud and Vladimir Drinfeld.
\newblock The wave front set of the {F}ourier transform of algebraic measures.
\newblock {\em Israel J. Math.}, 207(2):527--580, 2015.

\bibitem[AG12]{AG3}
Avraham Aizenbud and Dmitry Gourevitch.
\newblock Multiplicity free {J}acquet modules.
\newblock {\em Canad. Math. Bull.}, 55(4):673--688, 2012.

\bibitem[AGS08]{AG2}
Avraham Aizenbud, Dmitry Gourevitch, and Eitan Sayag.
\newblock {$({\rm GL}_{n+1}(F),{\rm GL}_n(F))$} is a {G}elfand pair for any
  local field {$F$}.
\newblock {\em Compos. Math.}, 144(6):1504--1524, 2008.

\bibitem[Aiz13]{Aiz1}
Avraham Aizenbud.
\newblock A partial analog of the integrability theorem for distributions on
  {$p$}-adic spaces and applications.
\newblock {\em Israel J. Math.}, 193(1):233--262, 2013.

\bibitem[Ber84]{Ber1}
Joseph~N. Bernstein.
\newblock {$P$}-invariant distributions on {${\rm GL}(N)$} and the
  classification of unitary representations of {${\rm GL}(N)$}
  (non-{A}rchimedean case).
\newblock In {\em Lie group representations, {II} ({C}ollege {P}ark, {M}d.,
  1982/1983)}, volume 1041 of {\em Lecture Notes in Math.}, pages 50--102.
  Springer, Berlin, 1984.

\bibitem[BJ07]{wondComp2}
Armand Borel and Lizhen Ji.
\newblock Compactifications of symmetric spaces.
\newblock {\em J. Differential Geom.}, 75(1):1--56, 2007.

\bibitem[BP87]{wondComp1}
Michel Brion and Franz Pauer.
\newblock Valuations des espaces homog\`enes sph\'{e}riques.
\newblock {\em Comment. Math. Helv.}, 62(2):265--285, 1987.

\bibitem[Bri89]{sph5}
Michel Brion.
\newblock Spherical varieties: an introduction.
\newblock In {\em Topological methods in algebraic transformation groups ({N}ew
  {B}runswick, {NJ}, 1988)}, volume~80 of {\em Progr. Math.}, pages 11--26.
  Birkh\"{a}user Boston, Boston, MA, 1989.

\bibitem[BZ76]{BZ1}
I.~N. Bern\v{s}te\u{\i}n and A.~V. Zelevinski\u{\i}.
\newblock Representations of the group {$GL(n,F),$} where {$F$} is a local
  non-{A}rchimedean field.
\newblock {\em Uspehi Mat. Nauk}, 31(3(189)):5--70, 1976.

\bibitem[Cox95]{tor3}
David~A. Cox.
\newblock The homogeneous coordinate ring of a toric variety.
\newblock {\em J. Algebraic Geom.}, 4(1):17--50, 1995.

\bibitem[Dan78]{tor1}
V.~I. Danilov.
\newblock The geometry of toric varieties.
\newblock {\em Uspekhi Mat. Nauk}, 33(2(200)):85--134, 247, 1978.

\bibitem[DCP83]{onTheWondComp}
Corrado De~Concini and Claudio Procesi.
\newblock Complete symmetric varieties.
\newblock In Francesco Gherardelli, editor, {\em Invariant Theory}, pages
  1--44, Berlin, Heidelberg, 1983. Springer Berlin Heidelberg.

\bibitem[Del02]{sym1}
Patrick Delorme.
\newblock Harmonic analysis on real reductive symmetric spaces.
\newblock In {\em Proceedings of the {I}nternational {C}ongress of
  {M}athematicians, {V}ol. {II} ({B}eijing, 2002)}, pages 545--554. Higher Ed.
  Press, Beijing, 2002.

\bibitem[Del10]{pAdicFinMul1}
Patrick Delorme.
\newblock Constant term of smooth {$H_\psi$}-spherical functions on a reductive
  {$p$}-adic group.
\newblock {\em Trans. Amer. Math. Soc.}, 362(2):933--955, 2010.

\bibitem[DS11]{sym2}
Patrick Delorme and Vincent S\'{e}cherre.
\newblock An analogue of the {C}artan decomposition for {$p$}-adic symmetric
  spaces of split {$p$}-adic reductive groups.
\newblock {\em Pacific J. Math.}, 251(1):1--21, 2011.

\bibitem[Ful93]{tor2}
William Fulton.
\newblock {\em Introduction to toric varieties}, volume 131 of {\em Annals of
  Mathematics Studies}.
\newblock Princeton University Press, Princeton, NJ, 1993.
\newblock The William H. Roever Lectures in Geometry.

\bibitem[GK75]{GK}
I.~M. Gelfand and D.~A. Kazhdan.
\newblock Representations of the group {${\rm GL}(n,K)$} where {$K$} is a local
  field.
\newblock In {\em Lie groups and their representations ({P}roc. {S}ummer
  {S}chool, {B}olyai {J}\'{a}nos {M}ath. {S}oc., {B}udapest, 1971)}, pages
  95--118, 1975.

\bibitem[Gou10]{Dima1}
Dmitry Gourevitch.
\newblock {\em Multiplicity {O}ne {T}heorems and {I}nvariant {D}istributions}.
\newblock ProQuest LLC, Ann Arbor, MI, 2010.
\newblock Thesis (Ph.D.)--The Weizmann Institute of Science (Israel).

\bibitem[H\"03]{Hor}
Lars H\"{o}rmander.
\newblock {\em The analysis of linear partial differential operators. {I}}.
\newblock Classics in Mathematics. Springer-Verlag, Berlin, 2003.
\newblock Distribution theory and Fourier analysis, Reprint of the second
  (1990) edition [Springer, Berlin; MR1065993 (91m:35001a)].

\bibitem[HC73]{HC}
Harish-Chandra.
\newblock Harmonic analysis on reductive {$p$}-adic groups.
\newblock In {\em Harmonic analysis on homogeneous spaces ({P}roc. {S}ympos.
  {P}ure {M}ath., {V}ol. {XXVI}, {W}illiams {C}oll., {W}illiamstown, {M}ass.,
  1972)}, pages 167--192, 1973.

\bibitem[Hei85]{wf}
D.~B. Heifetz.
\newblock {$p$}-adic oscillatory integrals and wave front sets.
\newblock {\em Pacific J. Math.}, 116(2):285--305, 1985.

\bibitem[KKS15]{sph6}
Friedrich Knop, Bernhard Kr\"{o}tz, and Henrik Schlichtkrull.
\newblock The local structure theorem for real spherical varieties.
\newblock {\em Compos. Math.}, 151(11):2145--2159, 2015.

\bibitem[Kno96]{wondComp3}
Friedrich Knop.
\newblock Automorphisms, root systems, and compactifications of homogeneous
  varieties.
\newblock {\em J. Amer. Math. Soc.}, 9(1):153--174, 1996.

\bibitem[KO13]{sphFinMul1}
Toshiyuki Kobayashi and Toshio Oshima.
\newblock Finite multiplicity theorems for induction and restriction.
\newblock {\em Adv. Math.}, 248:921--944, 2013.

\bibitem[Kob14]{sym8}
Toshiyuki Kobayashi.
\newblock Symmetric pairs with finite-multiplicity property for branching laws
  of admissible representations.
\newblock {\em Proc. Japan Acad. Ser. A Math. Sci.}, 90(6):79--83, 2014.

\bibitem[KS16]{sphFinMul2}
Bernhard Kr\"{o}tz and Henrik Schlichtkrull.
\newblock Multiplicity bounds and the subrepresentation theorem for real
  spherical spaces.
\newblock {\em Trans. Amer. Math. Soc.}, 368(4):2749--2762, 2016.

\bibitem[KS18]{sph8}
Bernhard Kr\"{o}tz and Henrik Schlichtkrull.
\newblock Harmonic analysis for real spherical spaces.
\newblock {\em Acta Math. Sin. (Engl. Ser.)}, 34(3):341--370, 2018.

\bibitem[Lun01]{sph3}
Dominique Luna.
\newblock Vari\'{e}t\'{e}s sph\'{e}riques de type {$A$}.
\newblock {\em Publ. Math. Inst. Hautes \'{E}tudes Sci.}, (94):161--226, 2001.

\bibitem[Per14]{sph1}
Nicolas Perrin.
\newblock On the geometry of spherical varieties.
\newblock {\em Transform. Groups}, 19(1):171--223, 2014.

\bibitem[Pez10]{sph4}
Guido Pezzini.
\newblock Lectures on spherical and wonderful varieties.
\newblock {\em Les cours du CIRM}, 1:33--53, 01 2010.

\bibitem[Sak08]{sph11}
Yiannis Sakellaridis.
\newblock On the unramified spectrum of spherical varieties over {$p$}-adic
  fields.
\newblock {\em Compos. Math.}, 144(4):978--1016, 2008.

\bibitem[Sak13]{sph10}
Yiannis Sakellaridis.
\newblock Spherical functions on spherical varieties.
\newblock {\em Amer. J. Math.}, 135(5):1291--1381, 2013.

\bibitem[SV17]{pAdicFinMul2}
Yiannis Sakellaridis and Akshay Venkatesh.
\newblock Periods and harmonic analysis on spherical varieties.
\newblock {\em Ast\'{e}risque}, (396):viii+360, 2017.

\bibitem[Tim11]{sph2}
Dmitry~A. Timashev.
\newblock {\em Homogeneous spaces and equivariant embeddings}, volume 138 of
  {\em Encyclopaedia of Mathematical Sciences}.
\newblock Springer, Heidelberg, 2011.
\newblock Invariant Theory and Algebraic Transformation Groups, 8.

\bibitem[vdB87]{sym5}
Erik~P. van~den Ban.
\newblock Invariant differential operators on a semisimple symmetric space and
  finite multiplicities in a {P}lancherel formula.
\newblock {\em Ark. Mat.}, 25(2):175--187, 1987.

\bibitem[vdB88]{sym3}
E.~P. van~den Ban.
\newblock The principal series for a reductive symmetric space. {I}.
  {$H$}-fixed distribution vectors.
\newblock {\em Ann. Sci. \'{E}cole Norm. Sup. (4)}, 21(3):359--412, 1988.

\bibitem[vdB92]{sym4}
E.~P. van~den Ban.
\newblock The principal series for a reductive symmetric space. {II}.
  {E}isenstein integrals.
\newblock {\em J. Funct. Anal.}, 109(2):331--441, 1992.

\bibitem[vdBCD96]{sym6}
Erik~P. van~den Ban, Jacques Carmona, and Patrick Delorme.
\newblock Paquets d'ondes dans l'espace de {S}chwartz d'un espace
  sym\'{e}trique r\'{e}ductif.
\newblock {\em J. Funct. Anal.}, 139(1):225--243, 1996.

\bibitem[vdBS01]{sym7}
Erik van~den Ban and Henrik Schlichtkrull.
\newblock Harmonic analysis on reductive symmetric spaces.
\newblock In {\em European {C}ongress of {M}athematics, {V}ol. {I}
  ({B}arcelona, 2000)}, volume 201 of {\em Progr. Math.}, pages 565--582.
  Birkh\"{a}user, Basel, 2001.

\end{thebibliography}
\bibliographystyle{alpha}

\end{document}